\documentclass[11pt]{article}

\usepackage[utf8]{inputenc}
\usepackage[T1]{fontenc}
\usepackage[main=english,ngerman]{babel}
\usepackage{lmodern}
\usepackage{tcolorbox}
\usepackage{enumerate}
\usepackage{multirow}
\usepackage{nicefrac}
\usepackage{caption}
\usepackage{graphicx}
\usepackage{tikz}
\usetikzlibrary{patterns}
\usetikzlibrary{arrows}
\usepackage{tikz-3dplot}
\usepackage{subfigure}
\usepackage{pgfplots}
\usepackage{hyperref}
\usepackage{picture}
\usepackage[absolute,overlay]{textpos}
\usepackage{algorithm}
\usepackage{algpseudocode}
\usepackage{cite}

\pgfplotsset{
  log x ticks with fixed point/.style={
      xticklabel={
        \pgfkeys{/pgf/fpu=true}
        \pgfmathparse{exp(\tick)}%
        \pgfmathprintnumber[fixed relative, precision=3]{\pgfmathresult}
        \pgfkeys{/pgf/fpu=false}
      }
  },
  log y ticks with fixed point/.style={
      yticklabel={
        \pgfkeys{/pgf/fpu=true}
        \pgfmathparse{exp(\tick)}%
        \pgfmathprintnumber[fixed relative, precision=3]{\pgfmathresult}
        \pgfkeys{/pgf/fpu=false}
      }
  }
}

\usepackage{geometry}
\usepackage{float}

\usepackage{amsmath,amssymb,amsthm}
\usepackage{mathtools}
\usepackage[mathscr]{euscript}
\usepackage{xpatch}

\providecommand{\keywords}[1]
{
  \small	
  \textbf{\textit{Keywords---}} #1
}

\newtheorem{definition}{Definition}
\theoremstyle{remark}
\newtheorem{remark}{Remark}
\theoremstyle{theorem}

\theoremstyle{theorem}
\newtheorem{lemma}{Lemma}

\newtheorem{assumption}{Assumption}

\usepackage[rightcaption]{sidecap}

\renewcommand{\o}{\Omega}

\newcommand{\p}{\boldsymbol}
\newcommand{\f}{\boldsymbol}

\graphicspath{{Images/}}

\begin{document}
	
	\title{Smooth multi-patch scaled boundary isogeometric analysis for Kirchhoff-Love shells}
    \author{Mathias Reichle$^1$\footnote{Corresponding author: reichle@lbb.rwth-aachen.de \qquad \qquad \qquad \qquad Preprint submitted to Springer}, Jeremias Arf$^2$, Bernd Simeon$^2$ and Sven Klinkel$^1$}
    \date{ $^1$ Chair of Structural Analysis and Dynamics, RWTH Aachen, Germany \\
    $^2$ Departement of Mathematics, RPTU Kaiserslautern-Landau, Germany}
	\maketitle
	\vspace{-0.9cm}
\begin{abstract}
In this work, a linear Kirchhoff-Love shell formulation in the framework of scaled boundary isogeometric analysis is presented that aims to provide a simple approach to trimming for NURBS-based shell analysis. 
To obtain a global $C^1$-regular test function space for the shell discretization, an inter-patch coupling is applied  with adjusted basis functions in the vicinity of the scaling center to ensure the approximation ability. Doing so,  the scaled boundary geometries are related to the concept of analysis-suitable $G^1$ parametrizations.This yields a coupling of patch boundaries in a strong sense that is restricted to $G^1$-smooth surfaces. The proposed approach is advantageous to trimmed geometries due to the incorporation of the trimming curve in the boundary representation that provides an exact representation in the planar domain. The potential of the approach is demonstrated by several problems of untrimmed and trimmed geometries of Kirchhoff-Love shell analysis evaluated against error norms and displacements. Lastly, the applicability is highlighted in the analysis of a violin structure including arbitrarily shaped patches. 
  \end{abstract}

\section{Introduction}
In modern engineering applications such as automotive engineering and aerospace engineering, computer-aided design (CAD) is a key method to design structural components. The mathematical underlying of these drawings usually consists of B-splines and non-uniform rational B-splines (NURBS) and provides that by using isogeometric analysis (IGA) \cite{Hughes2005} for the numerical analysis, the model for the design and the analysis are the same. A major issue of numerical analyses is the treatment of trimmed geometries whereas parts of the initial geometry are cut out \cite{Marussig2018}. This implies, that a numerical analysis in common frameworks is not conductible anymore in a straightforward manner. Several approaches are presented in the literature for treating trimming such as \cite{Coradello2020hierarchically,Leidinger2019,Coradello2020} mostly considering an approximation of the trimming curve and a limitation to quadrilateral sections.\\
Isogeometric analysis is especially suitable for Kirchhoff-Love shells \cite{Kiendl2009} since the formulation is derived on the change of the normal vector which requires well-defined second order derivatives of the basis functions. This is naturally fulfilled within single patches of quadratic or higher order. However, across patch boundaries, IGA patches are not naturally $C^{p-1}$ continuous. This implies that the Kirchhoff-Love shell formulation is not applicable for multi-patch structures in IGA. As complex shell structures often consist of multi-patch geometries, $C^1$ continuity needs to be enforced at patch boundaries. Common approaches to ensure higher continuity on multiple patches are Nitsche methods \cite{Guo2015,benzaken2021}, mortar methods \cite{dornisch2015,Chasapi2020}, and penalty methods \cite{breitenberger2015,herrema2019,Coradello2021,proserpio2022}, among others \cite{kiendl2010,schuss2019}. However, all of these methods are coupling techniques in a weak sense that can only represent the isogeometric discretization space approximately $C^1$ smooth. To ensure exact $C^1$ smooth multi-patch geometries, a special class of $G^1$ continuous multi-patch surfaces, namely analysis-suitable $G^1$ surfaces \cite{Collin2016}  can be utilized in the two-dimensional domain. An approach of analysis-suitable $G^1$ spaces that will be utilized in this work, is presented in \cite{Farahat2022} for Kirchhoff-Love shells with restriction to quadrilateral patches. 

A promising approach to take the exact boundary of geometries more into account is the scaled boundary isogeometric analysis (SB-IGA) \cite{Natarajan2015,Arioli2019}. It combines the advantages of the scaled boundary finite element method \cite{song1997,Wolf2000} and the isogeometric analysis providing patches of arbitrarily shaped boundaries. Since the discretization of SB-IGA patches does not only consist of multi-patches for trimmed patches but also in a single IGA-patch domain, a coupling approach is inherently necessary to obtain $C^1$-continuity on the SB-IGA domain.

This contribution presents a Kirchhoff-Love shell in the framework of scaled boundary isogeometric analysis with analysis-suitable $G^1$ surfaces that are especially suitable for complex, trimmed geometries by incorporating the trimming curve into the boundary representation of the geometry in a simple way. Moreover, to preserve the approximation ability of the SB-IGA test functions, special scaling center test functions have to be introduced. A step that can also be observed if smooth polar splines are considered \cite{toshniwal2017}. The proposed approach is tested on its general performance by untrimmed shell structures and the potential is later outlined by trimmed shell structures in the linear case. The advantages are summarized as follows.
\begin{itemize}
    \item A linear Kirchhoff-Love shell formulation is presented in isogeometric boundary representation.

    \item A patch coupling of smooth patches across boundary edges is applied utilizing the analysis-suitable $G^1$ concept including modifications in the vicinity of the scaling center.  

    \item Advantages of the approach to discretized star-shaped domains of an arbitrary number of boundaries are outlined.

    \item Examples of discretizations for complex topologies are presented and evaluated. 
\end{itemize}

The work is structured as follows. In Section \ref{Kirchhoff}, the linear formulation of the Kirchhoff-Love theory is provided. Section \ref{subsection:B-splines_and_SB-IGA} presents the fundamental concept of SB-IGA and the notation herein. Further, Section \ref{Section:PlanarC1} discusses the choice of basis functions for $C^1$ continuity in the planar domain and the application to Kirchhoff-Love shells with the incorporation of trimming. The proposed approach is tested in numerical examples in Section \ref{Section:NumExa}. Finally, the contribution is concluded in Section \ref{Section:ConOut} including an outlook to further research.

\section{Kirchhoff-Love shell formulation}\label{Kirchhoff}
The Kirchhoff-Love shell formulation is recalled in this section, based on the derivations presented in \cite{Kiendl2009,krysl1996,benzaken2021}. The formulation is derived considering a single patch but can be extended to scaled boundary multi-patch structures as outlined in the following sections. Neglecting cross-sectional shear stresses, the Kirchhoff-Love shell formulation states the assumption that the initial director vectors to the shell center surface remain perpendicular during deformation. The description of the shell is reduced from the domain of the shell body to the mid-surface $\Omega \subset \mathbb{R}^3$ utilizing a convective covariant space $\tilde{\Omega} \subset \mathbb{R}^2$. Utilizing the unit normal vector on the mid-surface, the initial shell structure reduced on the mid-surface is described as
\begin{equation}
\mathbf{S}(\theta^1,\theta^2,\theta^3) = \mathbf{a}(\theta^1,\theta^2) + \theta^3 \mathbf{a}_3(\theta^1,\theta^2)
\end{equation}
where $\mathbf{a}$ is the position vector in the initial configuration, $\mathbf{a}_3$ is the unit normal vector on each point on the surface, and $\theta^3$ the thickness coordinate ranging from $-t/2$ to $+t/2$. Further, we use the Greek indices $\alpha,\beta, \lambda, \mu \in \{1,2\}$ and Latin indices $i,j \in \{ 1,2,3 \}$ and drop the dependence of $\theta^{\alpha}$.  To obtain the normal vector on the surface, the partial derivatives on the mid-surface are utilized to construct the unit normal vectors with 
\begin{equation}
    \mathbf{a}_{\alpha} = \mathbf{a}_{,\alpha} = \frac{\partial \mathbf{a}}{\partial \theta^{\alpha}} \qquad \mathbf{a}_3 = \frac{\mathbf{a}_1 \times \mathbf{a}_2}{|\mathbf{a}_1 \times \mathbf{a}_2|}.
\end{equation}
In the following, we will indicate differentiation w.r.t. the corresponding coordinate $(\cdot )_{,\alpha}= \frac{\partial}{\partial \theta^{\alpha}}$. Further, we define the $L_2$-norm of the normal vector $\mathbf{a}_3$ as $\bar{J}$. Utilizing the partial derivatives, we introduce the covariant metric coefficients constructed by 
\begin{equation}
    a_{\alpha\beta} = \mathbf{a}_\alpha \cdot \mathbf{a}_\beta.
\end{equation}
Due to the orthogonality condition, we construct the contravariant basis system
$\mathbf{a}^{\alpha} \cdot \mathbf{a}_{\beta} = \delta_{\beta}^{\alpha}$,
that yields a relation of the covariant and the contravariant metric system by the inverse operator $\{a^{\alpha\beta}\} = \{a_{\alpha\beta}\}^{-1}$, where $\{ \cdot \}$ denotes the matrix components. \\
In the discrete setting of the boundary conditions of the boundary value problem, we denote the boundary of the body $\Gamma= \partial \Omega$ with Dirichlet and Neumann boundary conditions $\Gamma_D$ and $\Gamma_N$ with the assumption that $\Omega$ is a smooth surface with well-defined derivatives of the curvature. The boundary conditions of $\Gamma_D$ and $\Gamma_N$ are non-overlapping $\Gamma = \overline{\Gamma_D \cup \Gamma_N}$. Further, we decompose the boundary conditions into $\Gamma_D = \Gamma_u \cup \Gamma_{\beta}$ and  $\Gamma_N = \Gamma_Q \cup \Gamma_M$ with $\Gamma_u \cap \Gamma_Q = \varnothing$ and $\Gamma_{\beta} \cap \Gamma_M = \varnothing$ since the prescribed displacements $\tilde{\mathbf{u}}$ and traction forces $\tilde{\mathbf{Q}}$, as well as the prescribed rotations $\tilde{\mathbf{\beta}}$ and traction moments $\tilde{\mathbf{M}}$ at the boundary are of energetically conjugate nature. A similar consideration is done at the corners $\chi \subset \Gamma$ with a split into Neumann boundary conditions $\chi_N \cap \Gamma_N$ and Dirichlet boundary conditions $\chi_D \cap \Gamma_D$.\\
Having defined the boundary, we choose some suitable discrete space $\mathbf{V}_h \in H^2(\Omega)$ in the sense of finite element methods and depending on the boundary conditions of the problem. Then, the discrete weak form seeks for $\mathbf{u}_h \in \mathbf{V}_h$ such that
\begin{subequations}
\begin{equation}
    a(\mathbf{u}_h,\mathbf{v}_h) = F(\mathbf{v}_h), \qquad \forall \; \mathbf{v}_h \in \mathbf{V}_h.
\end{equation}
Where the bilinear form $a$ and the linear form $F$ are expanded as 
\begin{equation}
 \label{eq:variational_form_0}
    a(\mathbf{u}_h,\mathbf{v}_h) = \int_{\Omega} \boldsymbol{\varepsilon}(\mathbf{v}_h) : \mathbf{n}(\mathbf{u}_h) \mathrm{d}\Omega + \int_{\Omega} \boldsymbol{\kappa}(\mathbf{v}_h) : \mathbf{m}(\mathbf{u}_h) \mathrm{d}\Omega.
\end{equation}
\begin{equation}
\label{eq:variational_form}
    F(\mathbf{v}_h) = \int_\Omega \mathbf{g} \cdot \mathbf{v}_h \mathrm{d}\Omega + \int_{\Gamma_Q} \tilde{\mathbf{Q}} \cdot \mathbf{v}_h \mathrm{d}\Gamma + \int_{\Gamma_M} \tilde{\mathbf{M}} \beta_n(\mathbf{v}_h) \mathrm{d}\Gamma + \sum_{e\in \chi_N} \left( \tilde{S} v_{3,h } \big|_e \right) 
\end{equation} 
\end{subequations}
where $\mathbf{g}$ is an applied body force, $\tilde{S}$ a prescribed twisting moment at each corner of $\chi_N$, and $\beta_n$ the normal rotation \cite{benzaken2021}. In this work, only body forces and traction forces are considered. 
Additionally, for a complete weak form, the kinematics and the constitutive relation are pending. The quantities $\boldsymbol{\varepsilon}$, $\boldsymbol{\kappa}$, $\mathbf{n}$ and $\mathbf{m}$ denote the membrane strains and bending strains depending on the test function $\mathbf{v}_h$ and the energetically conjugate membrane forces and bending moments depending on the discretized solution field $\mathbf{u}_h$. According to the Kirchhoff kinematical assumptions, transverse shear strains vanish and membrane and bending strains remain. Consequently, the decomposition of the linearized Green-Lagrange strain tensor reads
\begin{equation}
    \mathbf{E}_{lin}(\mathbf{v}_h) = \boldsymbol{\varepsilon}(\mathbf{v}_h) + \theta^3 \boldsymbol{\kappa}(\mathbf{v}_h).
\end{equation}
For a rigorous derivation of the linearization, see \cite{benzaken2021}. According to \cite{kaufmann2009} the linearized strain components are
\begin{subequations}
\begin{equation}
    \varepsilon_{\alpha \beta} = \frac{1}{2} \left(\mathbf{a}_{\beta} \cdot \mathbf{v}_{h,\alpha} +  \mathbf{a}_{\alpha} \cdot \mathbf{v}_{h,\beta}   \right)
\end{equation}
and
\begin{multline}
    \kappa_{\alpha \beta} = -\mathbf{a}_3 \cdot \mathbf{v}_{h,\alpha\beta} + \mathbf{a}_{\alpha,\beta} \cdot \mathbf{a}_3 \cdot \frac{1}{\Bar{J}} \left( \left(\mathbf{a}_2 \times \mathbf{a}_3 \right) \cdot \mathbf{v}_{h,1} -\left(\mathbf{a}_1 \times \mathbf{a}_3 \right) \cdot \mathbf{v}_{h,2} \right) \\ + \frac{1}{\Bar{J}} \left( \left(\mathbf{a}_{\alpha,\beta} \times \mathbf{a}_2 \right) \cdot \mathbf{v}_{h,1} -\left(\mathbf{a}_{\alpha,\beta} \times \mathbf{a}_1 \right) \cdot \mathbf{v}_{h,2} \right).
\end{multline}
\end{subequations}
Besides the kinematics, the stress resultants $\mathbf{n}$ and $\mathbf{m}$ are derived by applying a suitable constitutive relation. Thus, the second Piola-Kirchhoff stress tensor is analytically integrated through the thickness and reads as a decomposition into the stress tensors' components as  \cite{Coradello2021}
\begin{subequations}
\begin{equation}
    \mathbf{n}(\mathbf{u}_h) = t\cdot \mathbf{D} : \f{\varepsilon}(\mathbf{u}_h)
\end{equation}
and
\begin{equation}
    \mathbf{m}(\mathbf{u}_h) = \frac{t^3}{12} \cdot \mathbf{D} : \boldsymbol{\kappa}(\mathbf{u}_h)
\end{equation}
 with the materials fourth-order tensor $\mathbf{D}$. We want to remark, that the analytical integration herein assumes a constant thickness of the shell. Applying Hooke's law as a linear constitutive model, the tensor is  given as \cite{Coradello2021}
\begin{equation}
    \mathbf{D} = D^{\alpha\beta\lambda\mu} \mathbf{a}_{\alpha} \otimes \mathbf{a}_{\beta} \otimes \mathbf{a}_{\lambda} \otimes \mathbf{a}_{\mu} 
\end{equation}
and 
\begin{equation}
    D^{\alpha\beta\lambda\mu} = \frac{E}{2(1+\nu)} \left(a^{\alpha\lambda}a^{\beta\mu}+ a^{\alpha\mu}a^{\beta\lambda} + \frac{2\nu}{1-\nu}  a^{\alpha\beta}a^{\lambda\mu}\right).
\end{equation}
\end{subequations}
Herein the parameters $\nu$ and $E$ are the specific material parameters of Poisson's ratio and Young's modulus, respectively.  
It is remarkable, that due to the occurrence of second order derivatives of the basis functions in the bending components, $C^1$ continuity is required for the computation. This is fulfilled naturally for SB-IGA patches across the elements within each patch, however, at the patch boundaries, only $C^0$ continuity is given. Thus, $C^1$ continuity is enforced across patch boundaries in the manner of scaled boundary isogeometric analysis.
\section{B-splines and SB-IGA}
\label{subsection:B-splines_and_SB-IGA}
In this section, we introduce the notation herein and explain briefly some basic notions in the context of the SB-IGA ansatz. For more details regarding isogeometric analysis, we refer to \cite{IGA1,IGA3}. We start with the definition of B-spline functions and B-spline spaces, respectively.\\ 
An  increasing sequence of real numbers $\Xi \coloneqq \{ \xi_1 \leq  \xi_2  \leq \dots \leq \xi_{n+p+1}  \}$ for some $p \in \mathbb{N}$   is called \emph{knot vector}, where we assume  $0=\xi_1=\dots=\xi_{p+1}, \ \xi_{n+1}=\dots=\xi_{n+p+1}=1$, and call such knot vectors $p$-open. 
Further, the multiplicity of the $j$-th knot is denoted by $m(\xi_j)$.
Then  the univariate B-spline functions $\widehat{B}_{j,p}(\cdot)$ of degree $p$ corresponding to a given knot vector $\Xi$ is defined recursively by the \emph{Cox-DeBoor formula}:
\begin{align}
\widehat{B}_{j,0}(\zeta) \coloneqq \begin{cases}
1, \ \ \textup{if}  \ \zeta \in [\xi_{j},\xi_{j+1}) \\
0, \ \ \textup{else},
\end{cases}
\end{align}
\textup{and if }  $p \in \mathbb{N}_{\geq 1} \ \textup{we set}$ 
\begin{align}
\widehat{B}_{j,p}(\zeta)\coloneqq \frac{\zeta-\xi_{j}}{\xi_{j+p}-\xi_j} \widehat{B}_{j,p-1}(\zeta)  +\frac{\xi_{j+p+1}-\zeta}{\xi_{j+p+1}-\xi_{j+1}} \widehat{B}_{j+1,p-1}(\zeta).
\end{align}
Note  one sets $0/0=0$ to obtain  well-definedness. The knot vector $\Xi$ without knot repetitions is denoted by $\{ \psi_1, \dots , \psi_N \}$. \\
The  extension of the last spline definition to the multivariate case is achieved by a tensor product construction. In other words, we set for a given tensor knot vector   $\boldsymbol{\Xi} \coloneqq \Xi_1 \times   \dots \times \Xi_d $, where the $\Xi_{l}=\{ \xi_1^{l}, \dots , \xi_{n_l+p_l+1}^{l} \}, \ l=1, \dots , d$ are $p_l$-open,   and a given \emph{degree vector}   $\f{p} \coloneqq (p_1, \dots , p_d)$ for the multivariate case
\begin{align}
\widehat{B}_{\p{i},\f{p}}(\boldsymbol{\zeta}) \coloneqq \prod_{l=1}^{d} \widehat{B}_{i_l,p_l}(\zeta_l), \ \ \ \ \forall \, \p{i} \in \mathbf{I}, \ \  \boldsymbol{\zeta} \coloneqq (\zeta_1, \dots , \zeta_d),
\end{align}
with  $d$ as  the underlying dimension of the parametric domain $\widehat{\Omega}= (0,1)^d$ and $\mathbf{I}$ the multi-index set $ \mathbf{I}\coloneqq \{ (i_1,\dots,i_d) \  | \  1\leq i_l \leq n_l, \ l=1,\dots,d  \}$.\\
B-splines  fulfill several properties and for our purposes the most important ones are:
\begin{itemize}
	\item If  for all internal knots, the multiplicity satisfies $1 \leq m(\xi_j) \leq m \leq p, $ then the B-spline basis functions $\widehat{B}_{i,p}$ are global $C^{p-m}$-continuous. Therefore we define  in this case the regularity integer $r \coloneqq p-m$. Obviously, by the product structure, we get splines $\widehat{B}_{\p{i},\f{p}}$ which are $C^{r_l}$-smooth  w.r.t. the $l$-th coordinate direction if the internal multiplicities fulfill $1 \leq m(\xi_j^l) \leq m_l  \leq p_l, \ r_l \coloneqq p_l-m_l, \ \forall l \in 1, \dots , d$ in the multivariate case.    To  emphasize later the regularity of the splines, we introduce an upper index $r$ and write in the following $\widehat{B}_{{i},{p}}^{{r}}, \ \widehat{B}_{\p{i},\f{p}}^{\f{r}}$ respectively. Here  $\p{r} \coloneqq (r_1,\dots ,r_d) $.
	\item For univariate splines $\widehat{B}_{i,p}^r, \ p \geq 1, \ r \geq 0$ we have
	\begin{align}
	\label{eq:soline_der}
	\partial_{\zeta} \widehat{B}_{i,p}^r(\zeta) = \frac{p}{\xi_{i+p}-\xi_i}\widehat{B}_{i,p-1}^{r-1}(\zeta) +  \frac{p}{\xi_{i+p+1}-\xi_i}\widehat{B}_{i+1,p-1}^{r-1}(\zeta),
	\end{align} 
	with $\widehat{B}_{1,p-1}^{r-1}(\zeta)\coloneqq \widehat{B}_{n+1,p-1}^{r-1}(\zeta) \coloneqq 0$.
	\item  The support of the spline $\widehat{B}_{i,p}^r $ is part of the interval $[\xi_i,\xi_{i+p+1}]$.
\end{itemize}
The space spanned by all univariate splines $\widehat{B}_{i,p}^r$ corresponding to  given knot vector and degree $p$ and global regularity $r$  is denoted by $$S_p^r \coloneqq \textup{span}\{ \widehat{B}_{i,p}^r \ | \ i = 1,\dots , n \}.$$ 
 Later, to have more flexibility it could be useful to introduce a strictly positive weight function $W = \sum_{{i}} w_{{i}} \widehat{B}_{{i},{p}}^{{r}}  \in S^{r}_{p}$ and use  NURBS functions  $\widehat{N}_{{i},{p}}^{{r}} \coloneqq \frac{w_i \, \widehat{B}_{{i},{p}}^{{r}}}{W}$, the NURBS spaces ${N}_{p}^{r} \coloneqq \frac{1}{W}  S_{p}^{r}, $  respectively. Such  weighted spline  functions are needed for conic section parametrizations.
 For the multivariate case we just define the needed  spaces as  product spaces, e.g.  $$S_{p_1, \dots , p_d}^{r_1,\dots,r_d} \coloneqq S_{p_1}^{r_1} \otimes \dots \otimes S_{p_d}^{r_d} = \textup{span} \{\widehat{B}_{\p{i},\f{p}}^{\f{r}} \ | \  \p{i} \in \mathit{\mathbf{I}}  \},$$ provided proper univariate spaces.

To define discrete spaces in the computational domain  $\tilde{\o}$ we require a  parametrization  mapping $\mathbf{F} \colon \widehat{\Omega} \coloneqq (0,1)^d \rightarrow  \tilde{\o} \subset \mathbb{R}^d$ .
The knots stored in the knot vector $  \boldsymbol{\Xi} $, corresponding to  the underlying  splines, determine a mesh in the parametric domain $\widehat{\Omega} $, namely  $\widehat{M} \coloneqq \{ K_{\p{j}}\coloneqq (\psi_{j_1}^1,\psi_{j_1+1}^1 ) \times \dots \times (\psi_{j_{d}}^{d},\psi_{j_{d}+1}^{d} ) \ | \  \p{j}=(j_1,\dots,j_{d}), \ \textup{with} \ 1 \leq j_i <N_i\},$ and
with ${\boldsymbol{\Psi}}= \{\psi_1^1, \dots, \psi _{N_1}^1\}  \times \dots \times \{\psi_1^{d}, \dots, \psi _{N_{d}}^{d}\}$  \  \textup{as  the knot vector} \ ${\boldsymbol{\Xi}}$ \ 
\textup{without knot repetitions}.   
The image of this mesh under the mapping $\mathbf{F}$, i.e. $\mathcal{M} \coloneqq \{{\mathbf{F}}(K) \ | \ K \in \widehat{M} \}$, gives us a mesh structure in the physical domain. Obviously, by inserting knots without changing the parametrization  we can refine the mesh, which is known as  $h$-refinement; see \cite{Hughes2005,IGA1}.
For a mesh $\mathcal{M}$ we can define the global mesh size through $h \coloneqq \max\{h_{{K}} \ | \ {K} \in \widehat{M} \}$, where for ${K} \in \widehat{M}$ we denote with $h_{{K}} \coloneqq \textup{diam}({K}) $ the \emph{element size} and $\widehat{M}$ is the underlying parametric mesh.

The underlying concept of SB-IGA fits the fact that in CAD applications the computational domain is often described by means of its boundary. Furthermore, as discussed later the boundary representation is particularly suitable if trimming is considered.  \\  As long as the region of interest is star-convex we can choose a scaling center $\mathbf{z}_0 \in \mathbb{R}^d$ and the domain is then defined by a scaling of the boundary w.r.t. to $\mathbf{z}_0$. In this article, we restrict ourselves to planar SB domains, and in view of isogeometric analysis  we have some  boundary NURBS curve $\gamma(\zeta) = \textstyle\sum_{i=1}\mathbf{C}_i \ \widehat{N}_{i,p}^{r}(\zeta), \ \mathbf{C}_i \in \mathbb{R}^2 $ and define the SB-parametrization of some $\tilde{\o}$ through
\begin{equation*}
\mathbf{F} \colon \widehat{\o} \coloneqq (0,1)^2  \rightarrow \tilde{\o} \ , \ (\zeta,\xi) \mapsto \xi \ \big( \gamma(\zeta)-\mathbf{z}_0\big)+\mathbf{z}_0  \ \ (\textup{compare Fig. \ref{Fig:SB_illustration_1}}).
\end{equation*}

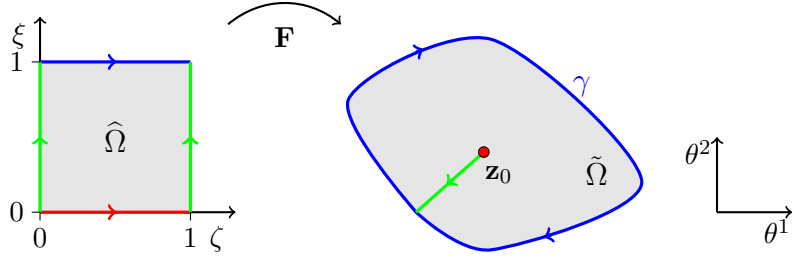
\begin{figure}
	\centering
	\begin{tikzpicture}
		\draw [fill=gray, opacity=0.2]  plot[smooth,dashed] coordinates {(5,0) (6,-0.5)  (8,0.4) (6,2.3) (4.1,1.5)  (5,0) };
	\draw[->, thick] (0,0) to (2.6,0);
	\draw[->, thick] (0,0) to (0,2.6);
	\draw[fill=gray, opacity=0.2] (0,0) rectangle (2,2);
	\draw[red, very thick] (0,0) --(2,0); 
	
	\draw[blue, very thick] (0,2) --(2,2);
	\draw[green, very thick] (0,0) --(0,2);  
	
	\draw[green, very thick] (2,0) --(2,2); 
	
		\draw [very thick, blue]  plot[smooth,dashed] coordinates {(5,0) (6,-0.5)  (8,0.4) (6,2.3) (4.1,1.5)  (5,0) };
			\draw[green, very thick] (5.9,0.8)--(5,0);
		\draw[fill=red] (5.9,0.8) circle(2pt);
		
	\draw[->, very thick, blue] (0.98,2) to (1.02,2); 
	\draw[->, very thick, red] (0.98,0) to (1.02,0); 
	
	\draw[->, very thick, green] (2,0.98) to (2,1.02); 
	
	\draw[->, very thick, green] (0,0.98) to (0,1.02); 
	
	\draw[->, very thick, green] (2,0.98) to (2,1.02); 
	
	\draw[->, very thick, green] (5.495,0.44) to (5.405,0.36); 
	
	\draw[->, very thick, blue] (5,1.98+0.15) to (5.1,2.02+0.15); 
	\draw[->, very thick, blue] (6.8,1.4-1.5-0.2) to (6.7,1.3-1.435-0.2); 
	
	\node[below] at (6.1,0.75) {$\mathbf{z}_0$};
	\node[below] at (2.35,-0.05) {$\zeta$};
	\node[left] at (-0.05,2.35) {$\xi$};
	
	\node[below] at (0,-0.08) {$0$};
		\node[left] at (-0.08,0) {$0$};
		\node[left] at (-0.08,2) {$1$};
		\node[below] at (2,-0.08) {$1$};
		\draw (-0.11,0) -- (0,0);
			\draw (0,-0.11) -- (0,0);
		\draw (-0.11,2) -- (0,2);
			\draw (2,-0.11) -- (2,0);
			\draw[->,thick, out=40,in=140] (2.5,2.5) to (4,2.5);
			
			\node[below] at (3.25,2.6) {$\mathbf{F}$};
			\node[above, blue, thick] at (7.2,1.4) {\large $\gamma$};
			
			\node at (1,1) {\large $\widehat{\Omega}$};
				\node at (7.4,0.55) {\large $\tilde{\Omega}$};
				
				\draw[->,thick] (9,0) --(10,0);
				\draw[->,thick] (9,0) --(9,1);
				\node at (9.8,-0.25) {$\theta^1$};
				\node at (8.75,0.8) {$\theta^2$};
	\end{tikzpicture}
	\caption{ \small Within SB-IGA a boundary description is used, where a well-chosen  scaling center is required. }
 \label{Fig:SB_illustration_1}
\end{figure}

\begin{remark}
In the considerations below it is allowed to replace the prefactor $\xi$ by a general degree $1$ polynomial $ c_1 \xi + c_2, \ c_1 >0, \ c_2\geq 0 $, i.e. no difficulties arise. This might be useful in some situations.
\end{remark}
 By the linearity w.r.t. the second parameter $\xi$  we can assume  for $\tilde{\o} \subset \mathbb{R}^2$ that  $\mathbf{F} \in \big[ {N}_p^r \otimes S_{p}^r  \big]^2$.  In particular, the weight function depends only on $\zeta$.

  \begin{figure}[H]	
  \centering
	\begin{tikzpicture}[scale=1.2]    
	    \node (eins) at (0.3,0.7) {\includegraphics[width=0.42\linewidth]{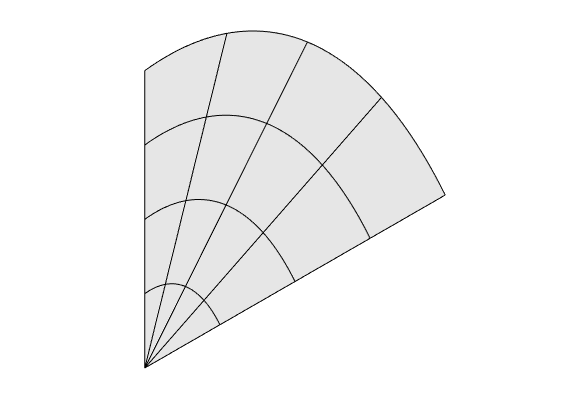}};
         \node (zwei) at (-0.45-0.3+6.6,0.7) {\includegraphics[width=0.42\linewidth]{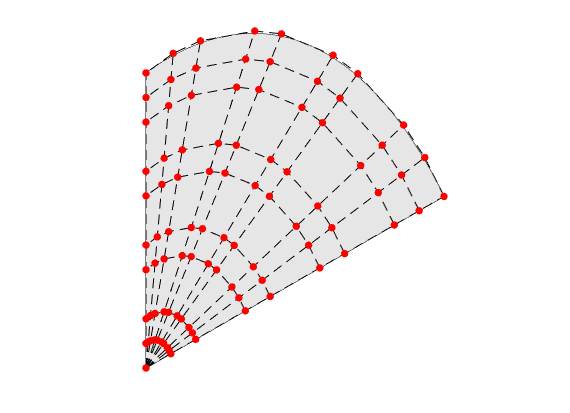}};
         \node[right] at (-1.8,-1.35) {\footnotesize (a) SB mesh};
               \node[right] at (-1.5+5.2,-1.35) {\footnotesize (b) Control points net};
          \draw[->,shift={(-1.017,-1)}] (0,0.1) -- (0,3.3);
          \draw[->,shift={(-0.9,-0.8)}] (-0.1,0) -- (2.4,0);
          \node at (-1.5,-0.8) {\small $(0,0)$};
            \node at (-1.22,2.3) {\small $\theta^2$};
             \node at (1.35,-0.95) {\small $\theta^1$};
             \node at (1.2,2) {\small $\tilde{\o}$};
	\end{tikzpicture}
\caption{\small The mesh and corresponding control net for a simple SB parametrization. Here we have $p=3, r=1$ for the underlying NURBS definition. }\label{Fig:Controls}
\end{figure}
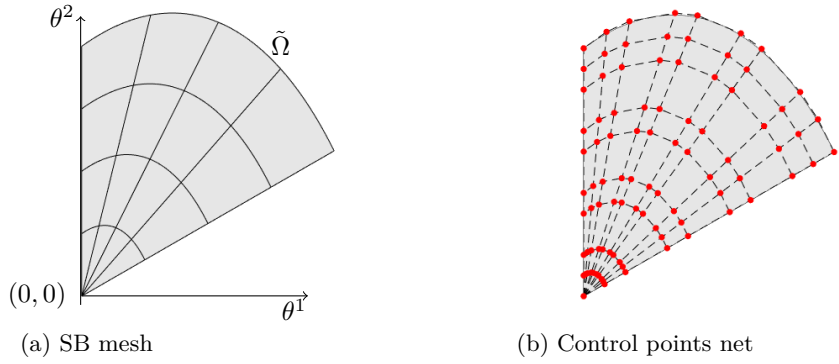
  
  We suppose that there are  control points $\mathbf{C}_{i,j} \in  \mathbb{R}^2$ associated to the NURBS $(\zeta,\xi) \mapsto \widehat{N}_{i,p}^r(\zeta) \widehat{B}_{j,p}^r(\xi) \in N_p^r \otimes S_p^r$ which define $\mathbf{F}$, namely $$\mathbf{F}(\zeta,\xi) = \sum_{i=1}^{n_1} \sum_{j=1}^{n_2}  \mathbf{C}_{i,j} \  \widehat{N}_{i,p}^r(\zeta) \widehat{B}_{j,p}^r(\xi) . $$ For simplification, we assume equal degree and regularity w.r.t. each coordinate direction. Due to the SB ansatz, we obtain in the physical domain $\tilde{\o}$ layers of control points and it is $\mathbf{C}_{1,1}=\mathbf{C}_{2,1}= \dots = \mathbf{C}_{n_1,1}$; cf. Fig. \ref{Fig:Controls}.   The isogeometric spaces utilized for discretization methods  are  defined as the push-forwards of the NURBS, namely $$\mathcal{V}_h= \mathcal{V}_h(r,p) \coloneqq \{ \hat{\phi} \circ \mathbf{F}^{-1} \ | \ \hat{\phi}  \in N_{p}^r \otimes {S}_p^r  \}.$$  In particular, we suppose the existence of an inverse mapping on the interior of $\tilde{\o}$. If the domain boundary $\partial\o$ is composed of different curves $\gamma^{(k)}$, one defines  parametrizations for each curve as written above and we get a multi-patch geometry; see Fig. \ref{Fig:SB_param_illustration_2}. To be more precise, for a $n$-patch geometry we have 	
\begin{align}
\bigcup_{m=1,\dots,n} \overline{\tilde{\o}_{m}} = \overline{\tilde{\Omega}}, \ \ \ \tilde{\o}_k \cap \tilde{\o}_l = \emptyset  \ \textup{if} \ k \neq l, \  \ \ \mathbf{F}_m \colon  \widehat{\Omega} \rightarrow \tilde{\o}_m, \ \mathbf{F}_m  \in \big[N_{p}^r \otimes {S}_p^r \big]^2 \  
\end{align} and $\mathbf{F}_m$ are SB parametrizations. IGA spaces in the multi-patch framework are straight-forwardly defined as $$\mathcal{V}^M_{h} \coloneqq \{ \phi  \colon \tilde{\o} \rightarrow \mathbb{R}\ | \phi_{|\tilde{\o}_m} \in \mathcal{V}_h^{(m)}, \ \forall m \},$$  where $\mathcal{V}_h^{(m)}$ denotes the IGA space  corresponding to the $m$-th patch, to $\mathbf{F}_m$, respectively.  If the boundary curves do not meet  $G^1$ but high global regularity is necessary, then the coupling is an issue. For all the patch coupling considerations we suppose the next assumption.
\begin{assumption}[Regular patch coupling] \color{white}{sdakjhjkh} 
\label{Assumption:coupling}\color{black}
    \begin{itemize}
        \item In each patch we use the NURBS and B-splines with the same degree $p$ and regularity $r>0$. Furthermore, the $\p{F}_m$ are $C^1$ in the interior of each patch with $C^1$-regular inverse.
        \item The control points at interfaces match, i.e. the control points of the meeting patches  coincide along the interface.
    \end{itemize}
\end{assumption}
Thus, it is justified to write for the set of parametric basis functions in the $m$-patch $$\{ \widehat{N}_{i,p}^r \cdot \widehat{B}_{j,p}^r \ | \ 1 \leq i \leq n_1^{(m)}, \ 1 \leq j \leq n_2\},$$ for proper $n_1^{(m)}, \ n_2 \in \mathbb{N}_{>1}.$\\
Below we look at the $C^1$ coupling of such SB-IGA patches, i.e. face spaces of the form
\begin{equation}
\mathcal{V}_h^{M,1} \subset \mathcal{V}_h^{M} \cap C^1(\tilde{\o}),
\end{equation} 
where the singularity of the $\mathbf{F}_m$ at $\mathbf{z}_0$  requires a special attention.

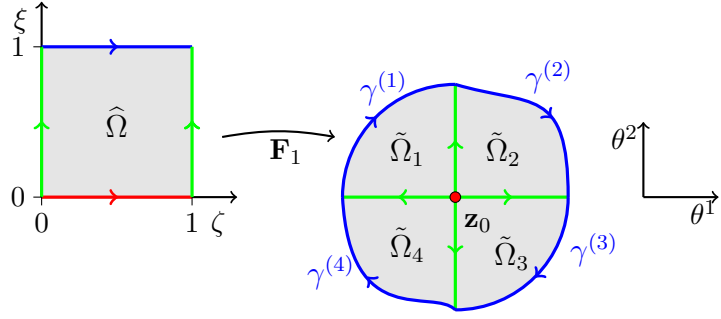
\begin{figure}
	\centering
	\begin{tikzpicture}
	\filldraw[fill opacity=0.2,fill=gray] (-1.5,0) -- (-1.5,1.5)  .. controls (-0.5,1.2) and (0,1.5) .. (0,0) ;

\filldraw[fill opacity=0.2,fill=gray] (-1.5,0) -- (-1.5,1.5) arc (90:180:1.5) -- cycle;

\filldraw[fill opacity=0.2,fill=gray] (-1.5,0) -- (-1.5,-1.5) .. controls (-1.8,-1.2) and (-2.9,-1.8) .. (-3,0) ;

\filldraw[fill opacity=0.2,fill=gray] (-1.5,0) -- (-1.5,-1.5) arc (270:360:1.5) -- cycle;

	\draw[very thick, green] (-1.5,0) -- (-3,0);
	\draw[very thick, green] (-1.5,0) -- (-1.5,1.5); 
	\draw[very thick, green] (-1.5,0) -- (-1.5,-1.5); 
	\draw[very thick, green] (-1.5,0) -- (0,0); 
	
	\draw[very thick, green,<-] (-0.74,0) -- (-0.76,0);
	\draw[very thick, green,->] (-1.5,0.74) -- (-1.5,0.76); 
	\draw[very thick, green,->] (-1.5,-0.74) -- (-1.5,-0.76); 
	\draw[very thick, green, <-] (-0.76-1.5,0) -- (-0.74-1.5,0);

	\draw[very thick, blue] (-1.5,1.5)  .. controls (-0.5,1.2) and (0,1.5) .. (0,0) ;
	\draw[very thick, blue] (-1.5,1.5) arc (90:180:1.5) ;
	\draw[very thick, blue] (-1.5,-1.5) .. controls (-1.8,-1.2) and (-2.9,-1.8) .. (-3,0) ;
	\draw[very thick, blue] (-1.5,-1.5) arc (270:360:1.5) ;
 
	\draw[fill=red] (-1.5,0) circle (2pt);

  	\draw[->, thick] (0-7,0) to (2.6-7,0);
  \draw[->, thick] (0-7,0) to (0-7,2.6);
  	\draw[fill=gray, opacity=0.2] (0-7,0) rectangle (2-7,2);
  \draw[red, very thick] (0-7,0) --(2-7,0); 
  
  \draw[blue, very thick] (0-7,2) --(-7+2,2);
  \draw[green, very thick] (0-7,0) --(0-7,2);  
  
  \draw[green, very thick] (2-7,0) --(2-7,2);

  	\node[below] at (6.1-7.3,-0.1) {$\mathbf{z}_0$};
  \node[below] at (2.35-7,-0.05) {$\zeta$};
  \node[left] at (-0.05-7,2.35) {$\xi$};
  
  \node[below] at (0-7,-0.08) {$0$};
  \node[left] at (-0.08-7,0) {$0$};
  \node[left] at (-0.08-7,2) {$1$};
  \node[below] at (2-7,-0.08) {$1$};
  \draw (-0.11-7,0) -- (0-7,0);
  \draw (0-7,-0.11) -- (0-7,0);
  \draw (-0.11-7,2) -- (0-7,2);
  \draw (2-7,-0.11) -- (2-7,0);

  \draw[->, very thick, blue] (0.98-7,2) to (1.02-7,2); 
  \draw[->, very thick, red] (0.98-7,0) to (1.02-7,0); 
  
  \draw[->, very thick, green] (2-7,0.98) to (2-7,1.02); 
  
  \draw[->, very thick, green] (0-7,0.98) to (0-7,1.02); 
  
  \draw[->, very thick, green] (2-7,0.98) to (2-7,1.02); 
  	\node at (1-7,1) {\large $\widehat{\Omega}$};
  	
  		\node at (-1.5-0.85*0.75,0.85*0.75) {\large $\tilde{\Omega}_1$};

	\node at (-1.5+0.85*0.75,0.85*0.75) {\large $\tilde{\Omega}_2$};
	
		\node at (-1.5+1*0.75,-1*0.75) {\large $\tilde{\Omega}_3$};
		
		\node at (-1.5-0.85*0.75,-0.85*0.75) {\large $\tilde{\Omega}_4$};

		\draw[->, very thick, blue] (-2.25-0.31-0.01,0.75+0.31-0.01) to (-2.25-0.31+0.01,0.75+0.31+0.01);

			\draw[->, very thick, blue,shift={(0.2,0)}] (-2.25-0.31-0.01+2.13,0.75+0.31+0.01) to (-2.25-0.31+0.01+2.13,0.75+0.31-0.01);

			\draw[->, very thick, blue,shift={(-0.1,0)}] (-2.25-0.31+0.01,0.75+0.31-0.01-2.13) to (-2.25-0.31-0.01,0.75+0.31+0.01-2.13); 
			
				\draw[->, very thick, blue] (-2.25-0.31+0.01+2.13,0.75+0.31+0.01-2.12) to (-2.25-0.31-0.01+2.13,0.75+0.31-0.01-2.12); 
				
	\node[left, blue] at (-2,1.45) {\large $\gamma^{(1)}$};
	\node[left, blue] at (0.2,1.6) {\large $\gamma^{(2)}$};
	\node[left, blue] at (0.8,-0.7) {\large $\gamma^{(3)}$};
	\node[left, blue] at (-2.7,-1) {\large $\gamma^{(4)}$};
	
		\draw[->,thick, out=10,in=170] (2.5-7.1,2.5-1.7) to (4-7.1,2.5-1.7);
		
		\node[below] at (3.25-7,2.6-1.7) {$\mathbf{F}_1$};

  \draw[->,thick,shift={(-8,0)}] (9,0) --(10,0);
 \draw[->,thick,shift={(-8,0)}] (9,0) --(9,1);
\node[shift={(-8,0)}] at (9.8,-0.2) {$\theta^1$};
\node[shift={(-8,0)}] at (8.75,0.8) {$\theta^2$};
	\end{tikzpicture}
	\caption{\small The boundary can be determined by the concatenation of several curves. In this situation the patch-wise defined discrete function spaces have to be coupled in order to obtain the desired global smoothness.}
 \label{Fig:SB_param_illustration_2}
\end{figure}

\section{From planar SB-IGA to KL shells}\label{Section:PlanarC1}
The appearance of second derivatives in the variational formulation \eqref{eq:variational_form_0} for the Kirchhoff-Love shell model requires a proper choice of the test and ansatz spaces for the computation of the shell displacements. An important aspect is the representation of the shell  via a middle surface which reduces the shell formulation to a $2D$   problem w.r.t. the parametric coordinates. We exploit this fact  by using  $C^1$-regular mappings in the parametric domain to express the shell configuration.

 \begin{figure}[H]	
  \centering
	\begin{tikzpicture}[scale=0.9]  
    \node (eins) at (0.3,-3.92) {\includegraphics[width=0.36\linewidth]{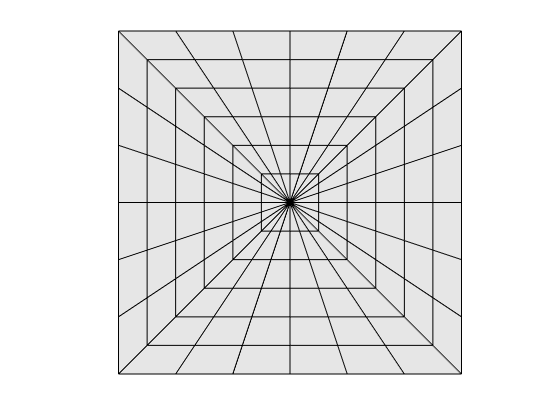}};
         \node (zwei) at (-0.45-0.3+7,-3.9) {\includegraphics[width=0.42\linewidth]{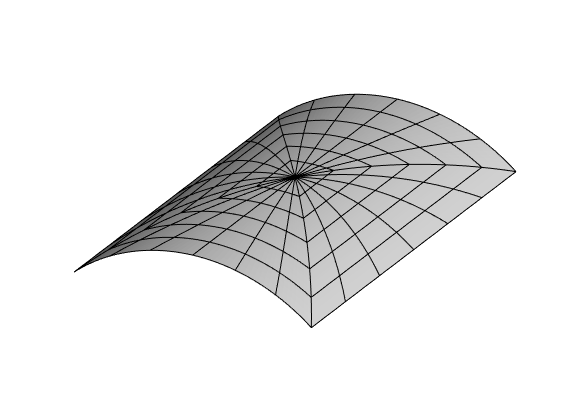}};
        \draw[->,thick,out=30,in=150,shift={(3,-3)}] (0,0) to (1,0);
		\node[below,shift={(2.7,-2.5)}] at (0.5,-0) {$\mathbf{R}$};
  \draw[->,thick,shift={(-1.45,-1.7)}] (0,-4) -- (4,-4);
   \draw[->,thick,shift={(-1.45,-1.7)}] (0,-4) -- (0,0);
     \node at (-0.45-0.3+3,-5.98) {$\theta^1$};
      \node at (-0.45-0.36-0.95,-1.9) {$\theta^2$};
       \node at (-0.45-0.3+9,-2.3) {$\o$};
    \node at (-0.45-0.3+2,-1.7) {$\tilde{\o}$};    
    	\draw[fill=gray, opacity=0.2,shift={(1,-4.4)}] (0-7,0) rectangle (2-7,2);
     	\draw[->, thick,shift={(1,-4.4)}] (0-7,0) to (2.3-7,0);
  \draw[->, thick,shift={(1,-4.4)}] (0-7,0) to (0-7,2.3);
  \draw[->,thick,out=30,in=150,shift={(-3,-3)}] (0,0) to (1,0);
  \node[below,shift={(-2.7,-2.7)}] at (0.5,-0) {$\mathbf{F}_m$};
  \node at (-4.9,-3.3) {$\widehat{\o}$};
    \node at (-3.8,-4.7) {$\zeta$};
    \node at (-6.2,-2.4) {$\xi$};
	\end{tikzpicture}
\caption{\small Our shell middle surface is parametrized by some mapping $\mathbf{R}$. We assume in next sections that the domain  $\tilde{\o}$ of this $\mathbf{R}$ is given as a SB multi-patch domain. }\label{Fig:shell_expl}
\end{figure}
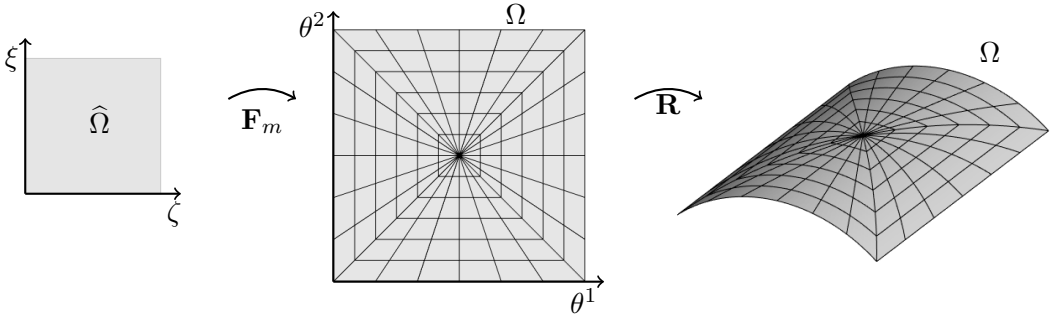
To be more precise, the basic idea for the KL-shell within SB-IGA framework is as follows. We express the initial shell configuration as well as the deformation of the shell utilizing the coupled test functions defined on the parametric domain $\tilde{\o}$, where we suppose  $\overline{\tilde{\o}}= \overline{\cup_m \tilde{\o}_m}$ to be given as a multi-patch SB-IGA geometry. On the patches $\tilde{\o}_m$ in turn we can introduce basis functions by means of the standard push-forwards. 
 This means our shell middle-surface can be  described by $$ \mathbf{R} \colon \tilde{\o} \rightarrow \o,$$
with \begin{equation}
\label{eq:Shell_param_form}
    \mathbf{R}^m = \sum_{i,j} \mathbf{P}^{(m)}_{i,j}  \big[ \widehat{N}_{i,p}^r \cdot \widehat{B}_{j,p}^r  \big] \circ {\mathbf{F}}_m^{-1}, \ \ \mathbf{R}^m=\mathbf{R}_{|\tilde{\o}_m}
\end{equation} for suitable control points $\mathbf{P}^{(m)}_{i,j} \in \mathbb{R}^3$; see Fig. \ref{Fig:shell_expl}. The advantage of such an approach is clear, namely coupling in the planar domain $\tilde{\o}$ is easier than a strong $C^1$ coupling of NURBS surface patches in $3D$. And the restriction to the special class of SB domains $\tilde{\o}$ gets useful if trimmed shells are considered.  
\begin{assumption}
The initial shell configuration is $C^1$ smooth in a strong sense. Further, we assume that the shell deformations can be expressed with $C^1$ mappings. In particular, we do not consider shells with kinks or deformations that lead to non-smooth shell configurations.
\end{assumption} 
Summarising it is enough to introduce globally $C^1$ basis functions on $\tilde{\o}$ to obtain test and ansatz functions for the Kirchhoff-Love shell model.\\
The coupling of  scalar SB-IGA spaces in planar domains is part of the next section. A component-wise generalization  to compute  shell displacements  is clear and will not be addressed separately.

\subsection{Planar SB-IGA with $C^1$ coupling}

\begin{figure}
	\centering
	\begin{tikzpicture}[scale = 0.8]
	\draw[very thick,blue, fill = gray , opacity = 0.2]  (0,0) .. controls(1,2) and (2,0) .. (3,0) to (1.5,-2.5) to (0,0);
	\draw[blue,thick]  (0,0) .. controls(1,2) and (2,0) .. (3,0) ;
	\draw[->, very thick, blue] (1.4+0.18,0.5+0.185+0.015) -- (1.41+0.18,0.49+0.187+0.015);
	\draw[thick,green] (0,0) to (1.5,-2.5);
	\draw[thick,green,<-] (0.75,-1.25) to (0.75+0.15,-1.25-0.25);
	\draw[thick,brown] (3,0) to (1.5,-2.5);
	\node[above, blue] at (1,0.85) {$\gamma^{(1)}$};  
	\node[above, blue] at (4.3,-1.6) {$\gamma^{(2)}$};

	\node at (1.5,-1.2) {\footnotesize $\tilde{\o}_1$};
 \node at (1.5,-0.6) {\footnotesize $\p{n}_{\Gamma}$};
 \node[brown] at (2.93,-0.55) {\footnotesize $\Gamma$};
	\node[right] at (3,0.1) {\footnotesize $\gamma^{(1)}(1)= \gamma^{(2)}(0)$};
	\node[below] at (1.5,-2.5) {\footnotesize $\mathbf{z}_0$};

	\draw[blue, fill = gray , opacity = 0.2]  (3,0) -- (4,-2) -- (1.5,-2.5) -- (3,0) ;
	\draw[blue,thick]  (3,0) -- (4,-2) ;
	\draw[blue,thick,->]  (3.5,-1) -- (3.5+0.1,-1-0.2) ;
	\draw[thick, green] (4,-2) -- (1.5,-2.5);
	\draw[thick, green,<-] (2.75+0.25,-2.25+0.05) -- (2.75,-2.25);
	
	\draw[thick, brown]  (1.5,-2.5) -- (3,0);
    \draw[very thick, ->] (2.5,-2.5/3) -- (2.5-5/4,-2.5/3+3/4); 
	\draw[thick, brown,->]  (2.25-0.15,-1.25-0.25) -- (2.25,-1.25);
	\node at (2.8,-1.6) {\footnotesize $\tilde{\o}_2$};

	\draw[fill = red] (1.5,-2.5) circle(2pt); 
	\draw[fill = black] (3,0) circle(2pt); 
   \draw[->,thick,shift={(-3.3,-2)}] (9,0) --(10,0);
 \draw[->,thick,shift={(-3.3,-2)}] (9,0) --(9,1);
\node[shift={(-2.55,-1.55)}] at (9.8,-0.37) {$\theta^1$};
\node[shift={(-2.6,-1.55)}] at (8.6,0.8) {$\theta^2$};
	\end{tikzpicture}
	\caption{ \small A simple two-patch SB geometry.}
	\label{Fig:2}
\end{figure}
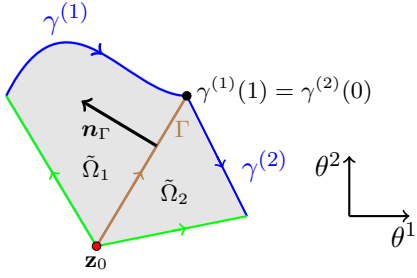

We explain the procedure of how to get $C^1$ basis functions in the two-patch situation as shown in Fig.  \ref{Fig:2} since a generalization to more patches is straightforward. According to the notation from the previous section and in view of isogeometric analysis we define the uncoupled basis functions on $ \overline{\tilde{\o}}= \overline{\tilde{\o}_1 \cup \tilde{\o}_2}$
$$ \mathcal{B} \coloneqq  \bigcup_m \{ \phi_{ij}^{(m)}  \ | \  i=1, \dots, n_1^{(m)}, \ j=1,\dots ,n_2\}, \ \ \textup{with} \  \phi_{ij}^{(m)} \circ \mathbf{F}_m = \widehat{N}_{i,p}^r \cdot \widehat{B}_{j,p}^r.$$ 
Above we extend the basis functions to the whole $\tilde{\o}$ by setting them to zero on the remaining patches. To obtain the necessary global $C^1$ smoothness we change the basis $\mathcal{B}$ according to the subsequent steps.
\subsubsection*{Remove basis functions near the scaling center}
 First, we remove in each patch all parametric and thus physical basis functions which are associated to $ \widehat{N}_{i,p}^r \cdot \widehat{B}_{j,p}^r$ with $j<3$.  This means we consider the modified basis $$\mathcal{B}^{'} \coloneqq  \bigcup_m \{ \phi_{ij}^{(m)}  \ | \  i=1, \dots, n_1^{(m)}, \ j=3,\dots ,n_2\}.$$
As a consequence, we see directly that the pull-backs of the remaining  basis functions are elements of $C^0(\overline{\widehat{\Omega}})$, i.e. we have a well-defined value $0$ in the singular point $\mathbf{z}_0$. Further, it can be easily verified that the pullbacks onto the parametric domain define  mappings 
$C^1(\overline{\widehat{\Omega}})$. Indeed, an application of the chain rule 
implies vanishing derivatives of the physical basis functions in the scaling center. For this purpose assume w.l.o.g. $\mathbf{z}_0 = \f{0}, \ \mathbf{F}(\zeta, \xi) =    \xi   \ \gamma(\zeta), \ \gamma = (\gamma_1,\gamma_2)^T$ and let $$\hat{\phi} \in \textup{span}\{ \widehat{N}_{i,p}^r \cdot \widehat{B}_{j,p}^r  \ | \ j\geq 3, r \geq 1\},$$ i.e. we have  a $C^1$ function with $\hat{\phi}= \partial_{\xi}\hat{\phi}= \partial_{\zeta}\hat{\phi} =0$ for $\xi=0$. \\
Set  $\phi(\theta^1,\theta^2) \coloneqq \hat{\phi} \circ \mathbf{F}^{-1}(\theta^1,\theta^2)$  then, with $(\theta^1,\theta^2)= \mathbf{F}(\zeta,\xi)$ it is 
\begin{align}
\label{eq:derivative_limit_1}
    \begin{bmatrix}
{\partial}_{\zeta} \hat{\phi}(\zeta,\xi) \\
\partial_{\xi}  \hat{\phi}(\zeta,\xi)
    \end{bmatrix} = \begin{bmatrix}
        \xi {\partial}_{\zeta}\gamma_1(\zeta)  & \xi {\partial}_{\zeta}\gamma_2(\zeta) \\
         \gamma_1(\zeta)  & \gamma_2(\zeta)
    \end{bmatrix} \begin{bmatrix}
        \partial_{\theta^1} \phi(\theta^1,\theta^2) \\
        \partial_{\theta^2} \phi(\theta^1,\theta^2)
    \end{bmatrix}.
\end{align}
 Below we use the abbreviation $$d(\zeta) \coloneqq \big({\partial}_{\zeta}\gamma_1(\zeta) \  \gamma_2(\zeta) -  {\partial}_{\zeta}\gamma_2(\zeta)\gamma_1(\zeta)\big),$$
and we have $d(\zeta) \neq 0$ due to the assumed invertibility of $\mathbf{F}$.
Hence,  it is
\begin{align}
\label{eq:derivativ limit_2}
   \begin{bmatrix}
        \partial_{\theta^1} \phi (\theta^1,\theta^2)\\
        \partial_{\theta^2} \phi(\theta^1,\theta^2)
   \end{bmatrix}  =  \frac{1}{\xi \ d(\zeta)}  \begin{bmatrix}
        \gamma_2(\zeta)  & -\xi {\partial}_{\zeta}\gamma_2(\zeta) \\
         -\gamma_1(\zeta)  & \xi {\partial}_{\zeta}\gamma_1(\zeta) 
    \end{bmatrix} 
     \begin{bmatrix}
{\partial}_{\zeta}  \hat{\phi} (\zeta,\xi)\\
\partial_{\xi}  \hat{\phi}(\zeta,\xi)
    \end{bmatrix}.
\end{align}
By linearity and the definition of the B-spline basis functions and the assumption \\${\hat{\phi} = \partial_{\xi}\hat{\phi}= \partial_{\zeta}\hat{\phi} =0, \ \textup{for} \ \xi=0 }$, we can w.l.o.g. suppose $\hat{\phi} (\zeta,\xi) = \widehat{N}(\zeta) \ \xi^2,$
for a suitable $\widehat{N}$. Note $ \widehat{B}_{j,p}^r(\xi) \in \mathcal{O}(\xi^2)$ for $\xi \rightarrow 0$ if $j\geq 3$.
For this case we study  the derivatives when  $\xi \rightarrow 0$. With \eqref{eq:derivativ limit_2} one sees
\begin{align}
\partial_{\theta^1} \phi(\theta^1,\theta^2) & = \frac{{\partial}_{\zeta} \widehat{N}(\zeta) \  \xi^2}{\xi \ d(\zeta)} \ \gamma_2(\zeta)- 2\frac{\xi  \ \widehat{N}(\zeta) \ \xi}{\xi \ d(\zeta)} \ {\partial}_{\zeta}\gamma_2(\zeta) \overset{ \xi \rightarrow 0}{\longrightarrow} \ 0 .
\end{align}
But this implies directly that $\phi(\theta^1,\theta^2)$ has a well-defined $\theta^1$-derivative in the scaling center, namely $\partial_{\theta^1} \phi = 0$ in $\mathbf{z}_0$. Analogously one gets the well-defined derivative $\partial_{\theta^2} \phi(\mathbf{z}_0) = 0$.\\
With this first basis modification step we avoid the problematic part near the scaling center.

\subsubsection*{Add scaling center basis functions}
 To preserve the approximation ability of SB-IGA test functions, we certainly have to introduce new test functions in the physical domain that determine the function value and derivatives at the scaling center. In the planar case, three additional test functions are sufficient, where we exploit the iso-parametric paradigm to define preliminary test functions $\phi_{i,sc} \in C^0(\o)$ with
\begin{align*}
\phi_{1,sc}(\mathbf{z}_0) = 1, \ \  \partial_{\theta^1}\phi_{1,sc}(\mathbf{z}_0) =  \partial_{\theta^2}\phi_{1,sc}(\mathbf{z}_0)= 0, \\
\phi_{2,sc}(\mathbf{z}_0) = 0, \ \  \partial_{\theta^1}\phi_{2,sc}(\mathbf{z}_0) = 1, \ \ \partial_{\theta^2}\phi_{2,sc}(\mathbf{z}_0)= 0, \\
\phi_{3,sc}(\mathbf{z}_0) = 0, \ \  \partial_{\theta^1}\phi_{3,sc}(\mathbf{z}_0) = 0, \ \ \partial_{\theta^2}\phi_{3,sc}(\mathbf{z}_0)= 1.
\end{align*}
Latter requirements can be easily satisfied if we use the entries of the geometry control points  as coefficients for the parametric pendants $\hat{\phi}_{i,sc}$. To be more precise, if we have
$\mathbf{F}_m = \sum_{j=1}^{n_2}  \sum_{i=1}^{n_1^{(m)}} \mathbf{C}_{i,j}^{(m)} \widehat{N}_{i,p}^r  \cdot \widehat{B}_{j,p}^r,$ then we set 
\begin{align}
\label{eq_sc_test_fun_param_1}
\hat{\phi}_{1,sc}^{(m)} \coloneqq  \sum_{j=1}^{p+1}\sum_{i=1}^{n_1^{(m)}}   \widehat{N}_{i,p}^r  \cdot \widehat{B}_{j,p}^r,  \\
\hat{\phi}_{2,sc}^{(m)} \coloneqq  \sum_{j=1}^{p+1}\sum_{i=1}^{n_1^{(m)}}  (\mathbf{C}_{i,j}^{(m)})_1 \  \widehat{N}_{i,p}^r  \cdot \widehat{B}_{j,p}^r, \\
\hat{\phi}_{3,sc}^{(m)} \coloneqq  \sum_{j=1}^{p+1}\sum_{i=1}^{n_1^{(m)}}  (\mathbf{C}_{i,j}^{(m)})_2  \ \widehat{N}_{i,p}^r  \cdot \widehat{B}_{j,p}^r.  \label{eq_sc_test_fun_param_3}
\end{align}
And then the $\phi_{i,sc}$ are determined  on $\tilde{\o}$ via $$(\phi_{i,sc})_{| \tilde{\o}_m }\circ \mathbf{F}_m= \hat{\phi}_{i,sc}^{(m)} .$$
By the properties of B-splines we see directly     $ {\phi}_{1,sc}^{(m)}=1, {\phi}_{2,sc}^{(m)}=\theta^1, \ {\phi}_{3,sc}^{(m)}=\theta^2 $ in a neighborhood of $\mathbf{z}_0$. In other words, we can choose the latter three functions for the determination of values and derivatives at $\mathbf{z}_0$, i.e. we add them to the set of basis functions used for the coupling step. Hence now we have the new set of basis functions 
$$\mathcal{B}^{''} \coloneqq \mathcal{B}^{'} \cup \{ \phi_{1,sc}, \ \phi_{2,sc}, \ \phi_{3,sc}\}.$$
The three new test functions are continuous since we assume that the patches match continuously. And we note that the $\phi_{i,sc}$ are only defined once for a scaling center.
For example, in Fig. \ref{Fig:scaling_center_funcs},  the scaling center test functions for a three-patch case are shown.

  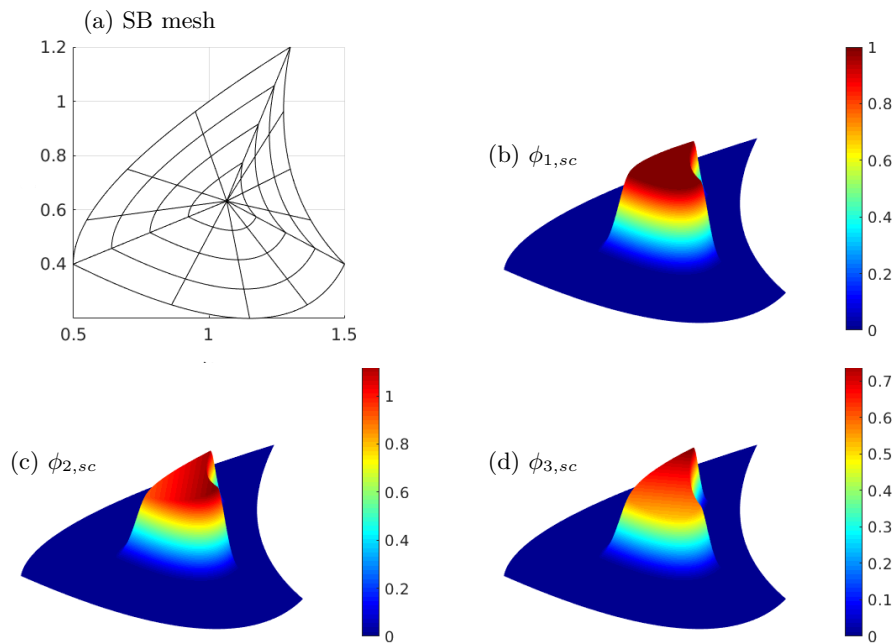
\begin{figure}[H]	
  \centering
	\begin{tikzpicture}[scale=0.85]  
    \node (eins) at (0.3,-4.1) {\includegraphics[width=0.42\linewidth]{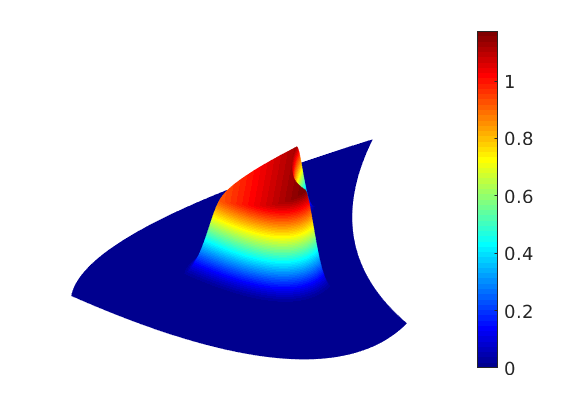}};
         \node (zwei) at (-0.45-0.3+8.6,-4.1) {\includegraphics[width=0.42\linewidth]{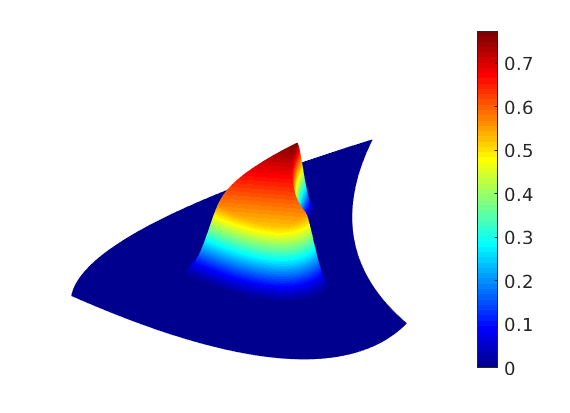}};
	    \node (eins) at (0.3,0.7) {\includegraphics[width=0.42\linewidth]{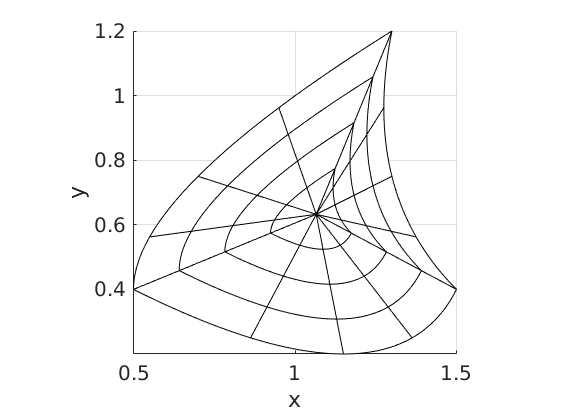}};         \node (zwei) at (-0.45-0.3+8.6,0.7) {\includegraphics[width=0.42\linewidth]{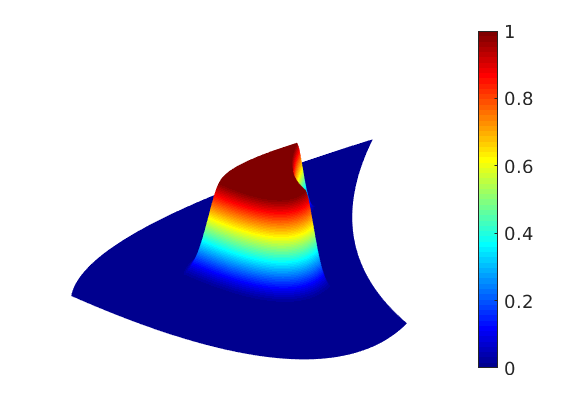}};
         \node at (-0.5,3.4) { \footnotesize (a) SB mesh};
        \node at (5.5,1.3) { \footnotesize (b) $\phi_{1,sc}$};
         \node at (-2,-3.5) { \footnotesize (c) $\phi_{2,sc}$};
          \node at (5.5,-3.5) { \footnotesize (d) $\phi_{3,sc}$};
          \filldraw[white] (0.45,-1.8) circle (4pt);
          \filldraw[white] (-2.4,0.9) circle (4pt);
          
	\end{tikzpicture}
\caption{ \small An illustration of the scaling center basis functions for a planar 3-patch example. Underlying (patch-wise) degree and regularity are $p=3, \ r=1$.}\label{Fig:scaling_center_funcs}
\end{figure}

\subsubsection*{The coupling step}
For the functions in $\mathcal{B}^{''}$, we now apply the needed coupling step to obtain the global $C^1$ regularity. This coupling is in turn composed of two simple substeps which correspond to the procedure used in \cite{Collin2016}.\\
Namely, we first couple the basis functions from $\mathcal{B}^{''}$  in a $C^0$ manner and observe that this is easily achieved due to the regular patch coupling Assumption \ref{Assumption:coupling}. So we have again a new basis $\mathcal{B}^{'''} = \mathcal{B}^{''} \cap C^0(\overline{\tilde{\o}})$. Looking at the interface $\Gamma$ between two patches, e.g. like in Fig. \ref{Fig:2},   we have for $\phi \in \mathcal{B}^{'''} $ a well-defined directional derivative in the direction of the interface. Thus it is enough to  find linear combinations $\phi$ of basis functions in $\mathcal{B}^{'''} $ that have continuous directional derivatives $\nabla \phi \cdot \p{n}_{\Gamma}$, where $\p{n}_{\Gamma}$ is a fixed normal vector to the interface.
Consequently, we get the $C^1$ coupled basis functions by the null space of the derivative jump matrix $M_{\Gamma}$, where 
$$(M_{\Gamma})_{kl} \coloneqq \int_{\Gamma} [[\nabla\phi_k \cdot \p{n}_{\Gamma}]] \,[[\nabla\phi_l \cdot \p{n}_{\Gamma}]] \, ds,  \ \phi_k, \phi_l \in \mathcal{B}^{'''} $$
with $[[g]]$ standing for the jump value of the $g$ across the interface and above we assume  $  \mathcal{B}^{'''} \coloneqq \{\phi_1, \ \phi_2,  \dots  \}$.
Straightforwardly this null space computation is adapted if several interfaces are involved.\\
Finally, we have with the basis of the null space of $M_{\Gamma}$ our $C^1$ smooth basis functions, i.e. our desired coupled basis $\mathcal{B}^1 \coloneqq  \mathcal{B}^{'''}  \cap C^1(\tilde{\o})$. We denote the space spanned by the coupled basis functions as $$\mathcal{V}_h^{M,1} = \textup{span}\{ \phi \ | \phi \in \mathcal{B}^{1}   \}.$$
\begin{remark}
The steps to get smooth basis functions above can be implemented directly. Even if the computation of the jump matrix for the different matrices seems problematic, we get well-defined integral values since the evaluations are   only done in quadrature points away from the singular point.
Besides, we note that all the steps from above  reduce to a suitable transformation matrix $T \in \mathbb{R}^{N \times N_1}$ expressing the coupled basis functions  w.r.t. the original basis. In particular $N_1 = \# \mathcal{B}^1, \ N = \# \mathcal{B}$. 
\end{remark}

After we showed the possibility to obtain $C^1$ mappings in the planar SB context, an obvious question arises.  Do  the coupled functions  yield a reasonable approximation behavior if we utilize them in numerical applications? It is well-known that a strong $C^1$ coupling might  lead to $C^1$ locking, i.e. the loss or worsening of approximation abilities. Admittedly, we do not give here a strict proof, of whether the $C^1$ coupling for planar SB domains   ensures optimal convergence  orders. However,  we want to outline why we expect the latter. Numerical tests confirm this conjecture. Following the results from \cite{Collin2016}, we get optimal convergence orders in  the planar multi-patch case, if we restrict ourselves to so-called analysis-suitable $G^1$ (AS-$G^1$)parametrizations  and B-spline basis functions with $p>r+1>1$. The mentioned parametrizations class is defined for the standard two-patch case, cf. Fig. \ref{Fig:1}, as follows.
\begin{definition}
In view of Fig. \ref{Fig:1} we set $\mathbf{F}^{(R)}=\mathbf{F}_2, \ \mathbf{F}^{(L)} = \mathbf{F}_1(\cdot +1,\cdot) $. The global parametrization $\tilde{\mathbf{F}}, $ i.e. $\tilde{\mathbf{F}}_{|\widehat{\o}^{(S)}}=\mathbf{F}^{(S)}, \ S \in \{L,R\}$, is analysis-suitable $G^1$ if there are polynomial functions $\alpha^{(S)}, \ \beta^{(S)} \colon [0,1] \rightarrow \mathbb{R}$ of degree at most $1$ s.t.
	\begin{align}
 \label{eq:ASG1_cond}
	\alpha^{(R)}(\xi) \, \partial_{\zeta}\mathbf{F}^{(L)}(0,\xi)- \alpha^{(L)}(\xi) 	\partial_{\zeta}\mathbf{F}^{(R)}(0,\xi)+ \beta(\xi) \, 	\partial_{\xi}\mathbf{F}^{(L)}(0,\xi) = \p{0},
	\end{align}
	with 
 $ \beta = \alpha^{(L)} \, \beta^{(R)} - \alpha^{(R)} \, \beta^{(L)}$. Further we require $\tilde{\mathbf{F}} \in C^0( \overline{\widehat{\Omega}^{(L)} \cup \widehat{\Omega}^{(R)}}; \mathbb{R}^2),  $ and that $ \tilde{\mathbf{F}}^{(S)} \in C^1(\overline{\widehat{\Omega}^{(S)}}; \mathbb{R}^2)$ is invertible.
\end{definition}
We call then the complete multi-patch parametrization analysis-suitable $G^1$, if the latter condition is fulfilled for each interface.
\begin{figure}[h!]
	\centering
		\begin{tikzpicture}[scale = 0.8]
		\draw[fill = gray , opacity = 0.2] (-5,-1.5) -- (-3,-1.5) -- (-3,0.5) -- (-5,0.5) -- (-5,-1.5);
		\draw[very thick,red] (-5,-1.5) -- (-3,-1.5); 
		\draw[very thick,green] (-3,-1.5) -- (-3,0.5); 
		
		\draw[very thick,blue] (-3,0.5) -- (-5,0.5);			
		\draw[->, out = 20, in = 150]  (-6,0.6) to (-0.5,0.8);
		\draw[->, out = -30, in = -180]  (-4,-1.6) to (1.3,-1.9);
		\node[above] at (-2,0.45) { ${\mathbf{F}}^{(L)}$};
		\node[above] at (-1.2,-2.15) { $\mathbf{F}^{(R)}$};
		\node at (-6,-0.8) { $\widehat{\Omega}^{(L)}$};
		\node at (-4,-0.8) { $\widehat{\Omega}^{(R)}$};
		\node at (1,0.25) { $\tilde{\o}^{(L)}$};
		\node at (3,-1.25) { $\tilde{\o}^{(R)}$};
		
		\draw[fill = gray , opacity = 0.2] (-7,-1.5) -- (-5,-1.5) -- (-5,0.5) -- (-7,0.5) -- (-7,-1.5);
		\draw[very thick,red] (-7,-1.5) -- (-5,-1.5); 
		\draw[very thick,green] (-7,0.5) -- (-7,-1.5); 
		\draw[very thick,blue] (-5,0.5) -- (-7,0.5);
		
		\draw[dashed,->] (-7,-1.5) to (-7,1);
		\draw[dashed,->] (-7,-1.5) to (-2,-1.5);
		\draw (-7,-1.5) -- (-7,-1.6);
		\draw (-5,-1.5) -- (-5,-1.6);
		\draw (-7,-1.5) -- (-7.1,-1.5);
		\draw (-7,0.5) -- (-7.1,0.5);
		\node[below] at (-7,-1.6) {  $-1$};
		\node[below] at (-5,-1.6) {  $0$};
		\node[left] at (-7.1,-1.5) {  $0$};
		\node[left] at (-7.1,0.5) {  $1$};
		\node[right] at (-7,0.8) {  $\xi$};
		\node[below] at (-2.3,-1.45) {  $\zeta$};
		\node[very thick,brown] at (2.5,0.5) { \small $\Gamma$};
		
		\draw[fill = gray, opacity = 0.2] (-0,1.5) -- (4,1.5) -- (2,-0.5) -- (0,-0.5)--(-0,1.5);
		\draw[fill = gray, opacity = 0.2] (2,-0.5) -- (4,1.5) -- (4,-2.5) -- (2,-2.5)--(2,-0.5);
		\draw[very thick,blue] (-0,1.5) -- (4,1.5);
		\draw[very thick,blue] (4,1.5) -- (4,-2.5);
		\draw[very thick,red] (-0,-0.5) -- (2,-0.5);
		\draw[very thick,red] (2,-0.5) -- (2,-2.5);
		\draw[very thick,green] (0,-0.5) -- (0,1.5);
		\draw[very thick,green] (2,-2.5) -- (4,-2.5);
		\draw[very thick,brown] (2,-0.5) -- (4,1.5);
		\draw[very thick,brown] (-5,0.5) -- (-5,-1.5);
    \draw[->,thick,shift={(-4.5+0.45,0)}] (9,0) --(10,0);
 \draw[->,thick,shift={(-4.5+0.45,0)}] (9,0) --(9,1);
\node[shift={(-3.65+0.45,-0.1)}] at (9.8,-0.2) {$\theta^1$};
\node[shift={(-3.6+0.3,0)}] at (8.75,0.8) {$\theta^2$};
		\end{tikzpicture}
	\caption{\small A standard planar two-patch geometry.}
		\label{Fig:1}
\end{figure}
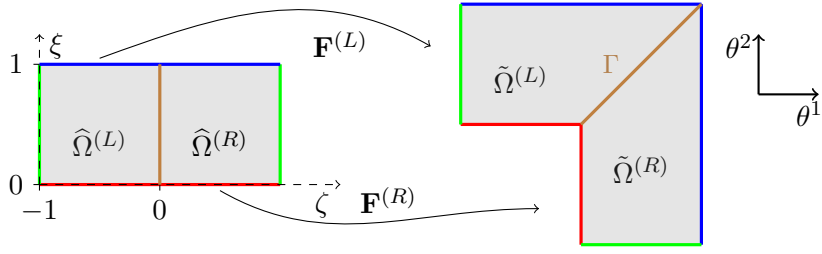
With the next lemma we relate the introduced AS-$G^1$ geometries to the planar SB ansatz.
\begin{lemma}[{SB-IGA patch coupling as quasi AS-$G^1$}]
       \label{Lemma:quasi_ASG1}
	The planar SB parametrizations  fulfill the condition in  \eqref{eq:ASG1_cond} for proper polynomials $\alpha^{(S)}, \ \beta^{(S)}$ of degree $\leq 1$.
\end{lemma}
\begin{proof}
This can be shown by simple calculations. But it is also clear by the observation that the SB-parametrizations can be approximated in a neighbourhood of an interface   by  bilinear parametrizations.  This can can be exemplified if w.l.o.g. $\mathbf{F}(\zeta,\xi) = \xi \gamma(\zeta)$, where the interface corresponds to the points with $\zeta=0$. Then it is $$\mathbf{F}(\zeta,\xi) = \xi \gamma(0) + \xi \zeta \, \partial_{\zeta} \gamma(0)+ \mathcal{O}(\zeta^2). $$
But since bilinear parametrizations are AS-$G^1$ as shown in \cite{Collin2016}, the assumption is clear.  
\end{proof}
Hence, if  we are away from the singular point we get with the SB ansatz  AS-$G^1$ parametrizations. And for the approximation near the scaling center we added the scaling center basis functions $\phi_{i,sc}$. This and Theorem 1 in \cite{Collin2016} suggest the feasibility of SB-IGA for $C^1$ coupling.

\subsection{Generalization to non-star shaped  and  trimmed geometries}
In this section, we shortly explain why the coupling from above can be generalized to more complicated geometries. \\
Up to now, the computational domain is considered to be star-shaped.  However, assuming we can partition our domain into several star-shaped blocks, we can do the coupling as explained previously block-wise. The coupling between the different blocks can then be realized analogously and easily if the block interfaces are given by straight lines. Moreover,  two patches from two blocks that meet at a straight interface meet w.l.o.g.  as two bilinear parametrizations, see Fig. \ref{Fig:non-star-shaped-domain}. But as already mentioned, the bilinear parametrizations are AS-$G^1$ meaning we should not see $C^1$ locking if $p>r+1>1$. 
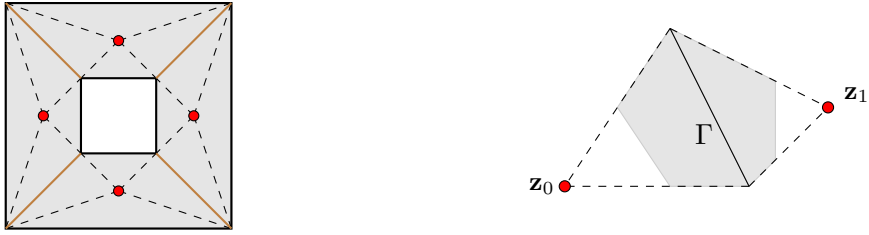
\begin{figure}
 \begin{minipage}{0.49\linewidth}
 \hspace{2cm}
    \begin{tikzpicture}[scale=1.0]
    \draw[black, thick, fill=gray, opacity = 0.2] (0,0) rectangle (3,3);
    \draw[black, thick, fill=white] (1,1) rectangle (2,2);
    \draw[black, thick] (0,0) rectangle (3,3);
    
    \draw[dashed] (0.5,1.5)--  (0,0);
	\draw[dashed] (0.5,1.5)--  (1,1);
	\draw[dashed] (0.5,1.5)--  (1,2);
	\draw[dashed] (0.5,1.5)--  (0,3);

    \draw[dashed] (1.5,2.5)--  (3,3);
	\draw[dashed] (1.5,2.5)--  (2,2);
	\draw[dashed] (1.5,2.5)--  (0,3);
	\draw[dashed] (1.5,2.5)--  (1,2);

    \draw[dashed] (2.5,1.5)--  (3,3);
	\draw[dashed] (2.5,1.5)--  (2,2);
	\draw[dashed] (2.5,1.5)--  (3,0);
	\draw[dashed] (2.5,1.5)--  (2,1);

    \draw[dashed] (1.5,0.5)--  (0,0);
	\draw[dashed] (1.5,0.5)--  (1,1);
	\draw[dashed] (1.5,0.5)--  (3,0);
	\draw[dashed] (1.5,0.5)--  (2,1);

    \draw[brown,thick] (0,0) to (1,1);
    \draw[brown,thick] (0,3) to (1,2);
    \draw[brown,thick] (3,3) to (2,2);
    \draw[brown,thick] (2,1) to (3,0);

    \draw[fill=red] (0.5,1.5) circle (1.9pt);
    \draw[fill=red] (1.5,2.5) circle (1.9pt);
    \draw[fill=red] (2.5,1.5) circle (1.9pt);
    \draw[fill=red] (1.5,0.5) circle (1.9pt);	
		\end{tikzpicture} 
 \end{minipage}
 \hspace{0cm}
 \begin{minipage}{0.49\linewidth}
		\centering
		\begin{tikzpicture}[scale=0.7]
		\draw (2,3) -- (3.5,0);
		\filldraw[fill=gray, opacity = 0.2] (2,3) --(3.5,0) -- (2,0) -- (1,1.5) -- (2,3);
		\filldraw[fill=gray, opacity = 0.2] (2,3) --(3.5,0) -- (4,0.5) -- (4,2) -- (2,3);
		\draw[dashed] (2,3) -- (0,0);
		\draw[dashed] (3.5,0) -- (0,0);
		\draw[dashed] (2,3) -- (5,1.5);
		\draw[dashed] (3.5,0) -- (5,1.5);
		\node at (2.65,1) {$\Gamma$};
		\node[left] at (0,0) { \small $\mathbf{z}_0$};
		\node[above] at (5.55,1.4) { \small  $\mathbf{z}_1$};
		\draw[fill = red] (0,0) circle(3pt); 
		\draw[fill = red] (5,1.5) circle(3pt); 
		\end{tikzpicture}
 \end{minipage}
		\caption{\small A non-star-shaped domain which is divided into star-shaped subdomains that have non-curved interfaces. Such multi-patch structures are still suitable for a $C^1$ coupling. If the interface between SB-param. with two different scaling centers is a straight line, then w.l.o.g. the patches meet as two bilinear patches and the param. is quasi AS-$G^1$, i.e. AS-$G^1$ except at the singular points.}
		\label{Fig:non-star-shaped-domain}
	\end{figure} 

The mentioned natural approach to applying a preliminary partition step to handle more complicated geometries is convenient when considering trimmed planar domains. The trimming operation, i.e. the cut away of domain parts is a standard operation within the IGA community and standing to reason for the description of different geometries. Especially, when dealing with perforated domains trimming curves are adequate to determine the computational domain. 
The concept of trimming can be formally  incorporated within planar SB-IGA. Namely, let us suppose an untrimmed domain $\tilde{\o}$ is defined as SB multi-patch domain. If a trimming curve $\gamma_T$ defines which parts are omitted, we first compute the intersections between the trimming curve and the boundary curves of the untrimmed domain. First, we think of a non-interior curve, i.e. there are at least two intersections.  After  possible knot insertions, it is possible to represent all relevant boundary segments by new NURBS curves exactly. If the trimmed domain is not-star shaped we first partition the new domain into star-convex blocks and then choose suitable scaling centers. Otherwise, we  choose directly a feasible scaling center.  \\
If there is an interior trimming curve, we   start with the partition step and then we detect the relevant boundary curves that are again inputs for the blocks-wise SB-parametrizations. Since the situation with an interior trimming curve is more interesting for us and the only one appearing in the numerics section later, we illustrate the proceeding for the latter case in the subsequent Fig. \ref{Fig:Trimming_3}.
\begin{figure}[H]
\centering
\begin{minipage}{0.32\textwidth}
	\begin{figure}[H]
		\centering
		\includegraphics[width=0.8\textwidth]{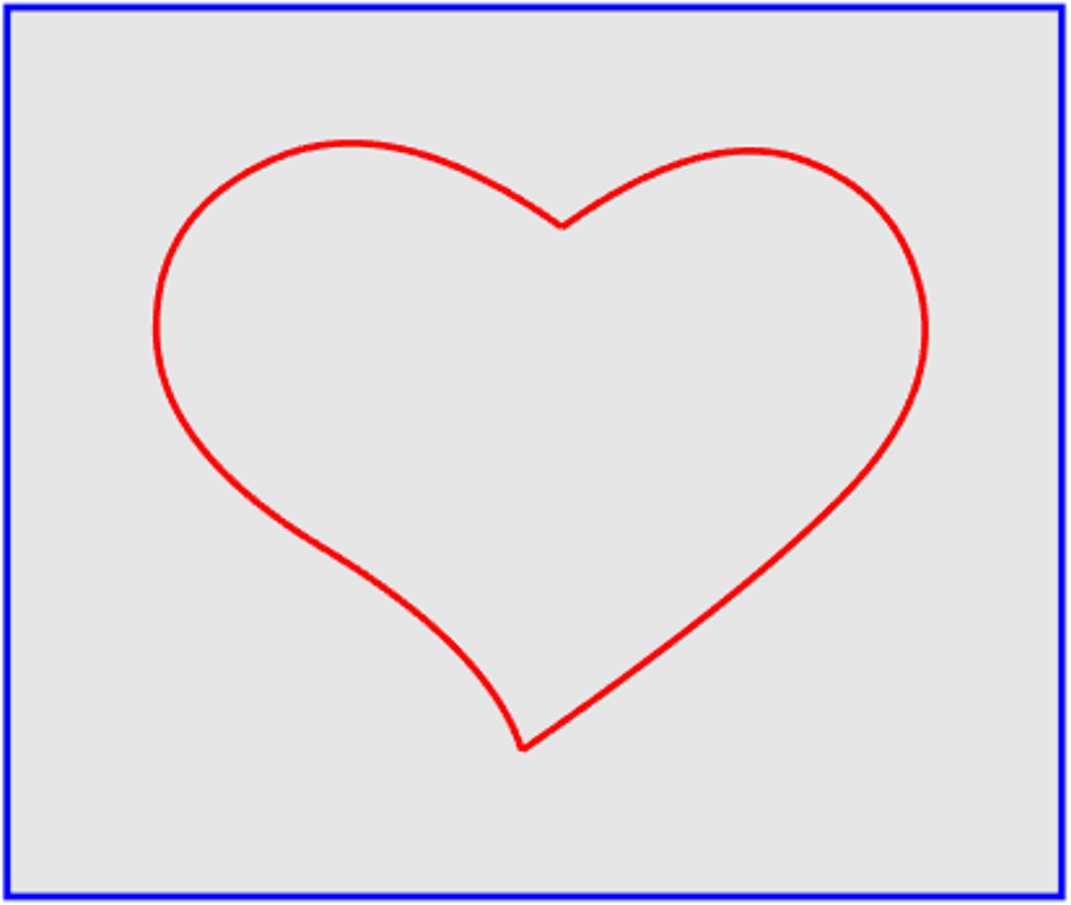}
		\caption*{\footnotesize{(a) Original domain with trimming curve $\gamma_T$}}
	\end{figure}
\end{minipage}
\begin{minipage}{0.32\textwidth}
	\begin{figure}[H]
		\centering
		\includegraphics[width=0.8\textwidth]{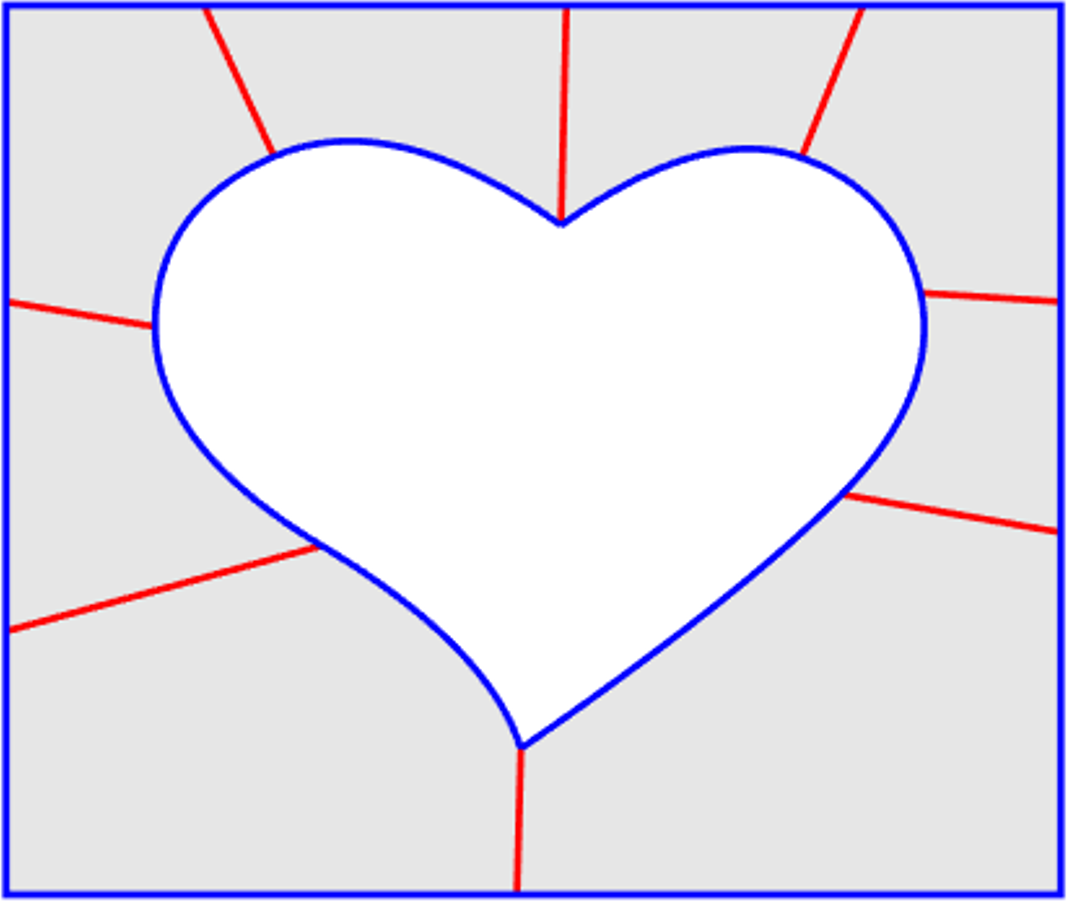}
		\caption*{\footnotesize{(b) Trimmed domain with cut lines.}}
	\end{figure}
\end{minipage}
\begin{minipage}{0.32\textwidth}
	\begin{figure}[H]
		\centering
		\includegraphics[width=0.8\textwidth]{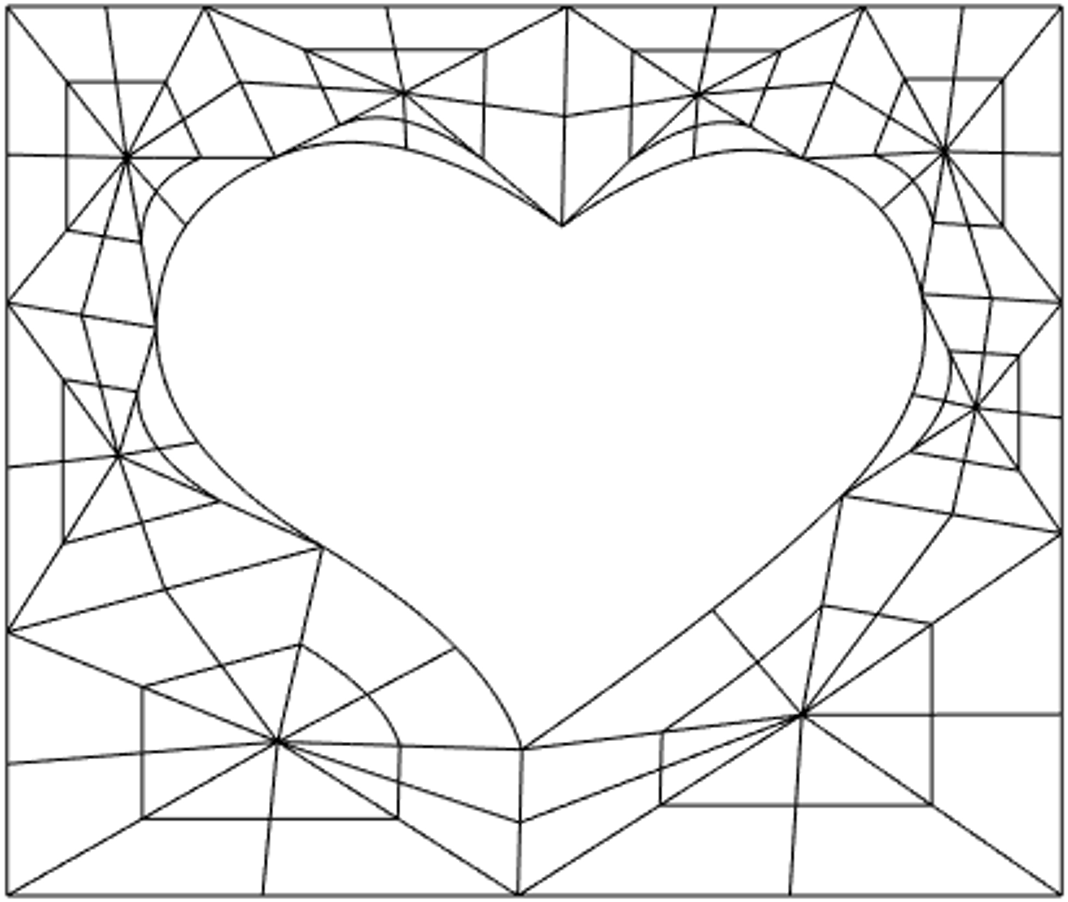}
		\caption*{\footnotesize{(c) Possible SB-mesh after the trimming.}}
	\end{figure}
\end{minipage}
    \caption{ \small Here we see a  trimming curve that would lead to a non-star domain. We can divide the trimmed domain into two star-shaped parts using a straight cut line.  }
    \label{Fig:Trimming_3}
\end{figure}
Before we come to numerical examples for the SB-IGA framework with $C^1$ coupling,  we briefly summarize the ideas for trimming in Alg. \ref{Algo}.
 \begin{algorithm}[H]
\small \caption{SB parametrization and trimming}\label{alg:cap}
\begin{algorithmic}[1]
\Require  Trimming curve $\gamma_T $; Untrimmed domain boundary curves $\gamma^{(i)}$
\State Compute intersection parameters $\zeta_k, s_k$ s.t.  $\gamma^{(i_k)} (\zeta_k)= \gamma_T(s_k)$
\If{there are no intersections} \Comment{Interior trimming curve}
    \State Divide trimmed domain into star-shaped blocks 
    \State Describe each boundary segment with NURBS curves  $\tilde{\gamma}^{(i)} \colon [0,1] \rightarrow \mathbb{R}^2$
    \State Compute suitable scaling centers 
\Else
    \State Check which boundary segments  define new boundary
     \State  Go to  line 3
\EndIf
\end{algorithmic}
\label{Algo}
\end{algorithm}

After we discussed how we can obtain $C^1$ basis functions on planar SB multi-patch domains, we exploit them  below to compute approximations to the Kirchhoff-Love shell model.

\section{Numerical Examples}\label{Section:NumExa}
In the following, the presented formulation is checked on its consistency at first by a simple test as square flat shell. In the further, the example of the Scordelis-Lo roof, known from the so-called 'shell obstacle course' is evaluated \cite{Belytschko1985,Macneal1985,Krysl2022}. Moreover, the double curved hyperbolic paraboloid \cite{Bathe2000} is evaluated as it considers negative Gaussian curvature and is a highly challenging example. Besides, the example of the Scordelis-Lo roof and further examples are evaluated in terms of trimmed structures as they can be found in literature \cite{Coradello2020,man2013,Coradello2020hierarchically} and are especially suitable for the boundary representation. Conducting an $h$-refinement for all examples, the models are validated by error norms or the displacements at specific points on reference solutions obtained analytically or from the literature. We denote that we further use $t$ as the thickness, $E$ as the Young's modulus, $\nu$ as the Poisson's ratio, and $h$ as mesh size for the underlying mesh with $1/h$ equidistant subdivisions with respect to both parametric coordinate directions. The approach has been implemented using MATLAB 2022 \cite{MATLAB:2022} in combination with the open source package GeoPDEs \cite{geopdesv3}. It is specifically designed to solve partial differential equations in the context of isogeometric analysis. We use the model from above, namely basis function in the space $\mathcal{V}_h^{M,1}$ for each component. But we note that due to the specific boundary conditions we have to reduce the number of test functions at specific places.

\subsection{Geometry approximation}\label{Geometry}
The coupling of the basis functions for the Kirchhoff-Love shell is done in the parametric domain $\tilde{\o}$ and $\mathbf{R}$ denotes the initial shell configuration parametrization. Sometimes the latter is not directly available in the form  \eqref{eq:Shell_param_form}. Then, due to the special structure of SB-IGA, it is not clear how to choose in general the control points in order to represent the initial shell configuration in an exact manner, to obtain $\mathbf{R}$ in the sense of \eqref{eq:Shell_param_form}, respectively. Thus we exploit the coupled $C^1$-smooth basis functions in the parametric domain to approximate the initial shell shape for the examples below. In this sense, we follow the isogeometric paradigm. To be more precise, if we state below the geometry mapping $\mathbf{R}$, we use in the actual computations an approximated $\mathbf{R}_{\textup{approx}} \approx \mathbf{R}$.   Doing so, we use a $L^2$-projection onto the test function space to get $\mathbf{R}_{\textup{approx}}$. Although this is a naive geometry approximation, it gives us even for coarse meshes reasonable results. Furthermore, several geometries like the hyperbolic paraboloid shape can be still represented exactly, except of rounding errors; see Fig. \ref{Fig:Para_approx} below.
\begin{figure}[H]
\begin{minipage}{0.5\textwidth}
	\begin{figure}[H]
		\centering
  \begin{tikzpicture}
		\node (myfirstpic) at (0,0) {\includegraphics[width=\textwidth]{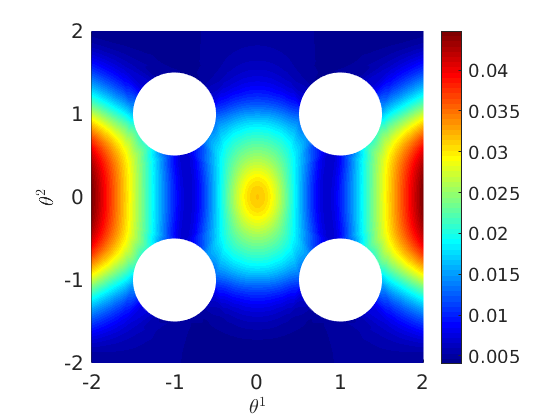}};
        \node at (-0.15,3) { \footnotesize{ $|\mathbf{R}_{\textup{approx}}(\theta^1,\theta^2)- \mathbf{R}(\theta^1,\theta^2)|$}};
  \end{tikzpicture}
		\caption*{\footnotesize{(a) Scordelis-Lo with 4 holes ($p=3, \, h=1/2$)}}
	\end{figure}
\end{minipage}
\begin{minipage}{0.5\textwidth}
	\begin{figure}[H]
		\centering
    \begin{tikzpicture}
		\node (myfirstpic) at (0,0) {\includegraphics[width=\textwidth]{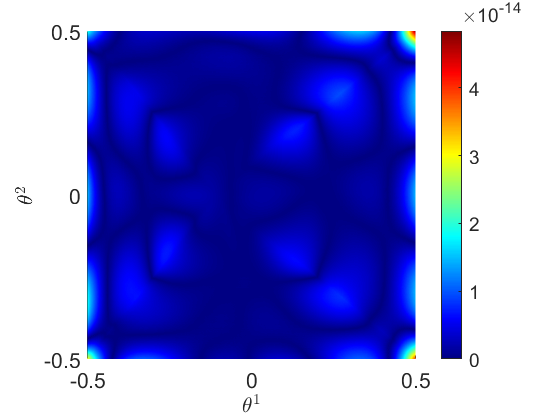}};
  \node at (-0.2,3) { \footnotesize $|\mathbf{R}_{\textup{approx}}(\theta^1,\theta^2)- \mathbf{R}(\theta^1,\theta^2)|$};
    \end{tikzpicture}
		\caption*{\footnotesize{(b) Hyperbolic paraboloid ($p=3, \, h=1/4$)}}
	\end{figure}
\end{minipage}
    \caption{ \small Here we see  the difference between the exact parametrization of the Scordelis-Lo roof with holes and the approximated parametrization mapping utilizing the coupled basis functions. To be more precise, the Euclidean norm difference  $|\mathbf{R}_{\textup{approx}}(\theta^1,\theta^2)- \mathbf{R}(\theta^1,\theta^2)|$ is plotted, where the underlying meshes are displayed in Fig. \ref{fig:Ex_5_Mesh} and Fig. \ref{fig:HypPara_Mesh}. Even for relatively coarse meshes we see only small differences and the geometry errors can be assumed to be small.}
    \label{Fig:Para_approx}
\end{figure}Furthermore, in principle, it is possible to use different NURBS for the initial geometry in order to have exact geometries. But, this requires more implementation effort due to the special form of the SB-IGA basis functions in view of quadrature rules. And on the other hand, we would leave in some sense the idea of \emph{iso}-geometric analysis.

\subsection{Smooth solution on a square shell}
The first example considering the shell formulation is a  flat shell of the square domain $\Omega = [0,2]^2$ under a smoothly distributed load. It is dedicated to investigating the convergence rates of the approach. The parametric domain coincides with the physical domain in this example. The properties are $E = 10^6$, $\nu =0.1$ and correspondingly $t = \left(12(1-\nu^2) /E\right)^{1/3}$. The boundary $\partial \Omega$ is fixed and a load $g$ is applied as $g_z = 4\pi^4\sin(\pi x) \sin(\pi y)$, such that the analytical solution of the deformation is $u_{z,ref}= \sin(\pi x) \sin(\pi y)$. The scaling center is placed in the middle of the structure and the mesh is evaluated for third, fourth and fifth order of basis functions. An exemplary mesh for $h=1/4$, $p=3,r=1$ for each patch and a corresponding deformation plot is shown in Fig. \ref{fig:Ex4_Mesh}. 
\begin{figure}[H] 
\centering
\scalebox{.9}{
\begin{minipage}{0.5\textwidth}
	\begin{figure}[H]
		\centering
		\includegraphics[width=\textwidth]{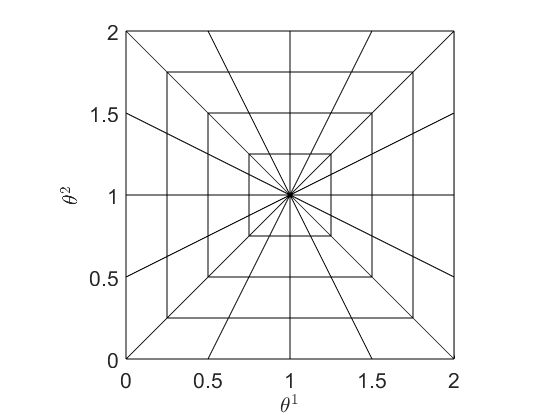}
		\caption*{\footnotesize{(a) Mesh }}
	\end{figure}
\end{minipage}
\begin{minipage}{0.5\textwidth}
	\begin{figure}[H]
		\centering
  \begin{tikzpicture}
      \node (myfirstpic) at (0,0) {\includegraphics[width=\textwidth]{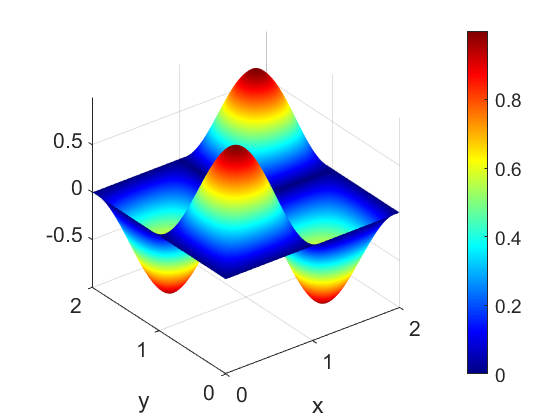}};
  \node at (2.1,0) {$|\mathbf{u}|$};
  \end{tikzpicture}
		\caption*{\footnotesize{(b) Deformed structure colored according to $|\mathbf{u}|$}}
	\end{figure}
\end{minipage}}
    \caption{\small Example of the smooth solution on a square shell. On the left, the underlying mesh of $h=1/4$ is pictured. On the right side, the corresponding deformed structure is shown with a plot of the deformation magnitude of the problem for $p=3$ and $r = 1$.}
    \label{fig:Ex4_Mesh}
\end{figure}
The results of the computation are compared to the analytical solution in terms of the $H^2$ seminorm and the $L^2$ norm defined as
\begin{flalign}
    | \mathbf{u} - \mathbf{u}_h |_{H^2(\Omega)} \quad \text{for the} \: H^2\text{-error} \\
    || \mathbf{u} - \mathbf{u}_h ||_{L^2(\Omega)} \quad \text{for the} \: L^2\text{-error}.
\end{flalign}
The convergence rates are shown in Fig. \ref{fig:Ex4Res}.
\begin{figure}[H]
\begin{minipage}{0.49\textwidth}
\begin{tikzpicture}
	\small \begin{loglogaxis} [
		width = \textwidth,
		xlabel={\footnotesize Mesh size $h$},
		ylabel={\footnotesize $|\mathbf{u} - \mathbf{u}_h|_{H^2(\Omega)}$},
		xmin=0.025, xmax=0.5,
		ymin=10^-7, ymax=2*10^0,
		xtick= {0.03,0.1,0.3,0.5},
		log x ticks with fixed point,
		ytick={10^-8,10^-7,10^-6,10^-5,10^-4,10^-3,10^-2,10^-1,10^0},
		legend pos=south east,
		legend columns = 3,
		xmajorgrids=true,
		ymajorgrids=true,
		grid style=dashed,
		]
		\addplot
		coordinates {
        (0.25,	1.622833492)
        (0.125,	0.357954789)
        (0.083333333,	0.156027042)
        (0.0625,	0.08664408)
        (0.05,	0.054904758)
        (0.041666667,	0.037829075)
        (0.035714286,	0.02761607)
        (0.03125,	0.021033008)
		};
		\addlegendentry{p=3, r=1},
		\addplot
		coordinates {
        (0.25,	0.185049305)
        (0.125,	0.023016661)
        (0.083333333,	0.006614898)
        (0.0625,	0.002728677)
        (0.05,	0.001377625)
        (0.041666667,	0.000786311)
		};
		\addlegendentry{p=4, r=1},
		\addplot[color=darkgray,mark=triangle*]
		coordinates {
        (0.25,	0.014239636)
        (0.125,	0.000873747)
        (0.083333333,	0.000171277)
        (0.0625,	5.40E-05)
        (0.05,	2.32E-05)
		};
		\addlegendentry{p=5, r=1},
		\addplot[color=blue, mark size=1.5pt,mark=none, dashed]
		coordinates {
			(0.25, 0.0625*10)
			(0.03125, 0.0009765625*10)};
		\addplot[color=red, mark size=1.5pt,mark=none, dashed]
		coordinates {
			(0.25, 0.015625*5)
			(0.03571429, 0.00004555395226*5)};
		\addplot[color=darkgray, mark size=1.5pt,mark=none, dashed]
		coordinates {
			(0.25, 0.00390625*2)
			(0.03571429, 0.000001626927062*2)};
		\legend{$p=3$,$p=4$,$p=5$,$\mathcal{O}(h^2)$,$\mathcal{O}(h^3)$,$\mathcal{O}(h^4)$}
	\end{loglogaxis}
\end{tikzpicture}
\caption*{\footnotesize{(a) Error in $H^2$ seminorm}}
\end{minipage}
\begin{minipage}{0.49\textwidth}
\begin{tikzpicture}
	\small \begin{loglogaxis} [
		width = \textwidth,
		xlabel={\footnotesize Mesh size $h$},
		ylabel={\footnotesize $||\mathbf{u} - \mathbf{u}_h||_{L^2(\Omega)}$},
		xmin=0.025, xmax=0.5,
		ymin=10^-13, ymax=10^-1,
		xtick={0.03,0.1,0.3,0.5},
		log x ticks with fixed point,
		ytick={10^-13,10^-12,10^-11,10^-10,10^-9,10^-8,10^-7,10^-6,10^-5,10^-4,10^-3,10^-2},
		legend pos=south east,
		legend columns = 3,
		xmajorgrids=true,
		ymajorgrids=true,
		grid style=dashed,
		]
		\addplot
		coordinates {
        (0.25,	0.009116595)
        (0.125,	0.000431354)
        (0.083333333,	8.1643E-05)
        (0.0625,	2.50139E-05)
        (0.05,	9.99218E-06)
        (0.041666667,	4.72578E-06)
        (0.035714286,	2.51199E-06)
        (0.03125,	1.45484E-06)
		};
		\addlegendentry{p=3, r=1},
		\addplot
		coordinates {
        (0.25,	0.000300825)
        (0.125,	9.55E-06)
        (0.083333333,	1.16E-06)
        (0.0625,	2.55E-07)
        (0.05,	7.89E-08)
        (0.041666667,	3.03E-08)
		};
		\addlegendentry{p=4, r=1},
		\addplot[color=darkgray,mark=triangle*]
		coordinates {
        (0.25,	1.28E-05)
        (0.125,	1.74E-07)
        (0.083333333,	1.43E-08)
        (0.0625,	4.30E-09)
        (0.05,	1.82E-09)
		};
		\addlegendentry{p=5, r=1},
		\addplot[color=blue, mark size=1.5pt,mark=none, dashed]
		coordinates {
			(0.25, 0.00390625*0.5)
			(0.03125, 0.0000009536743*0.5)
			};
		\addplot[color=red, mark size=1.5pt,mark=none, dashed]
		coordinates {
			(0.25, 0.0009765625*0.1)
			(0.03125, 0.0000000298*0.1)};
		\addplot[color=darkgray, mark size=1.5pt,mark=none, dashed]
		coordinates {
			(0.25, 0.000244140625*0.012)
			(0.03125, 0.0000000009313225746*0.012)};
        \legend{$p=3$,$p=4$,$p=5$,$\mathcal{O}(h^4)$,$\mathcal{O}(h^5)$,$\mathcal{O}(h^6)$}
	\end{loglogaxis}
\end{tikzpicture}
\caption*{\footnotesize{(b) Error in $L^2$ norm}}
\end{minipage}
\caption{\small Convergence studies of the $H^2$-errors (left) and the $L^2$-errors (right) on the example of the
smooth solution on a square shell and orders of $p = 3$, $p = 4$ and $p = 5$.}
\label{fig:Ex4Res}
\end{figure}

The convergence rates indicate  for both error estimations  optimal convergence rate $\mathcal{O}(h^{p-1})$ for the $H^2$ seminorm and $\mathcal{O}(h^{p+1})$ for the $L^2$ norm. The deviations for $p=5$ might be caused by large condition numbers together with rounding errors.

\subsection{Scordelis-Lo roof}
The Scordelis-Lo roof is a well-known model described by \cite{Scordelis1964}. It is a cylinder cutout of $80^{\circ}$ with dimensions of radius $R = 25$ and length $L =50$. The thickness is $t=0.25$ and the material parameters are $E = 4.32\cdot 10^8$ and $\nu = 0.0$. The structure is supported by rigid diaphragms at the curved edges and free at the straight edges. The Scordelis-Lo roof is investigated in three setups as they are the untrimmed structure, trimmed with an elliptic hole under gravity load, and trimmed with four holes under a single load. The untrimmed example is investigated to compare the boundary parametrization to standard IGA parametrizations. Further, the trimmed roofs are compared to numerical results obtained from the literature.
The exact parametrization function is chosen to
\begin{equation}
     \mathbf{R}(\theta^1,\theta^2) = \left[25\sin( \pi \theta^1 \frac{4}{9L_{\tilde{\Omega},\theta^1}}) \qquad \theta^2\frac{50}{L_{\tilde{\Omega},\theta^2}} \qquad 25\cos(\pi \theta^1 \frac{4}{9L_{\tilde{\Omega},\theta^1}}) \right]^T
\end{equation}
with $L_{\tilde{\Omega},\theta^1}$ and $L_{\tilde{\Omega},\theta^2}$ being the corresponding side length of $\tilde{\Omega}$ in each parametric direction.
\subsubsection*{Untrimmed Scordelis-Lo roof}
At first, the untrimmed roof under a gravity load of $g_z = 90$ per unit area is investigated and the reference solution is given as the vertical displacement at the midpoint of the straight, free edge as $u_{ref} = 0.3024$ \cite{Macneal1985}, \cite{Belytschko1985}. Since the reference considers shear deformation which is not represented in Kirchhoff-Love shells, the reference solution herein refers to \cite{Kiendl2009} with $u_{ref}=0.3006$ in terms of the vertical deflection at the same point of interest. The parametric mesh, defined on the domain $\tilde{\Omega} = [-0.5,0.5],[-1,1]$, and the corresponding initial shell configuration are shown in Fig. \ref{fig:ScoLo_Mesh}. The parametric mesh considers an offset of the scaling center with $\theta^1_{off} = 0.2$ and $\theta^2_{off} = 0.1$ to show the generality of the formulation.\\
\begin{figure}
\centering
\scalebox{.9}{
\hspace*{-0.95cm}
\begin{minipage}{0.49\textwidth}
	\begin{figure}[H]
		\centering
			\begin{tikzpicture}
			\node (myfirstpic) at (0,0) {\includegraphics[width=\textwidth]{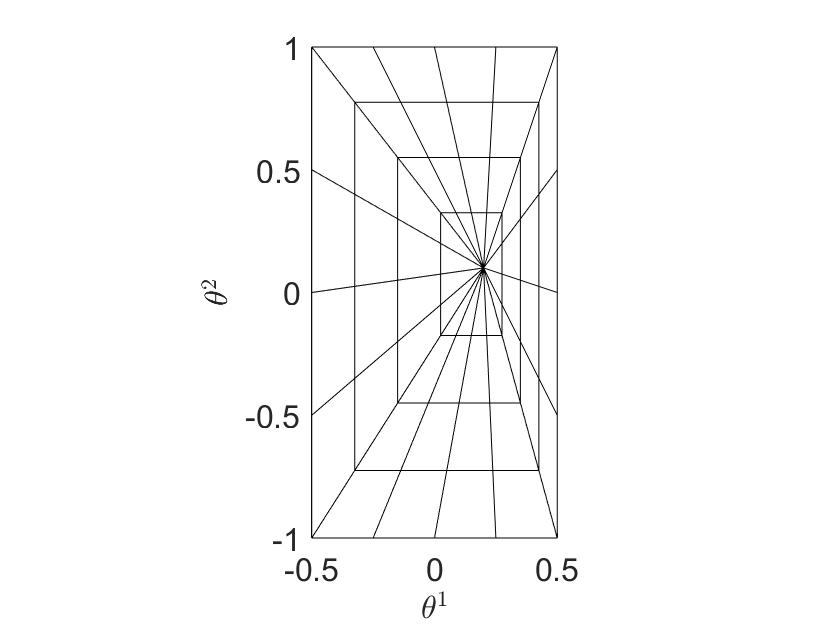}};
    \filldraw[red] (1.22,0.22) circle (2pt);
			\node at (0.15,2.7) {$\tilde{\o}$};
		\end{tikzpicture}
		\caption*{\footnotesize{(a) Parametric mesh}}
	\end{figure}
\end{minipage}\hspace*{-1cm}
\begin{minipage}{1cm}
\vspace{-3cm}
  \begin{tikzpicture}
      \draw[->,very thick, out=30,in=150] (-1,1) to (0,1);
      \node at (-0.5,1.4) {$\mathbf{R}$};
  \end{tikzpicture}
\end{minipage}
\begin{minipage}{0.49\textwidth}
	\begin{figure}[H]
		\centering
    \begin{tikzpicture}
	\node (myfirstpic) at (0,0) {\includegraphics[width=\textwidth]{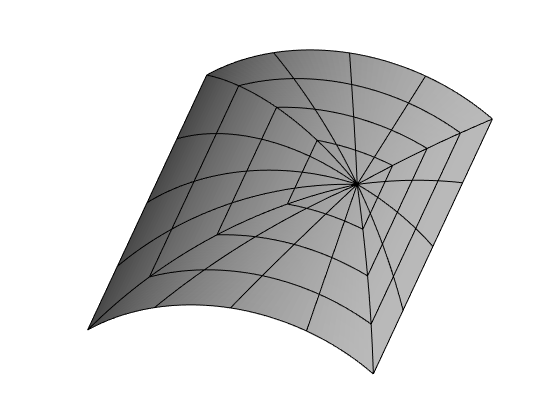}};
	\draw[|-|,thick] (1+0.2+0.2,-1.88-0.2-0.15) to (2.7+0.2+0.12,1.14-0.15+0.12);
	\node[rotate=65] at (2.4,-0.76) {$L=50$};
	\draw[thick,dashed] (-0.5,-3.1) -- (1.2,-1.88-0.25);
	 \draw [<->,thick] (-85:2.61)  arc (50:130:1) node [above,pos=0.45] {$\theta$};
	 \draw[thick,dashed] (-0.53,-3.1) -- (-2.55,-1.58);
  \filldraw[red] (2.04,0-0.47) circle (2pt);
	\end{tikzpicture}
		\caption*{\footnotesize{(b) Initial shell configuration}}
	\end{figure}
\end{minipage}}
    \centering
    \caption{\small Parametric mesh and initial shell configuration of the Scordelis-Lo roof for $h=1/2$.}
    \label{fig:ScoLo_Mesh}
\end{figure}
The results of basis functions of order $p=3$, $p=4$, and $p=5$ are shown in Fig. \ref{fig:ScoLo_Results} on the left side. Even though the scaling center is chosen in non-centered discretization, the results converge in a smooth way towards the reference solution. A comparison to a standard IGA computation of the Scordelis-Lo roof shows that the convergence of IGA-basis functions is at better results for fewer degrees of freedom. This is a natural consequence of the SB-IGA approach since four patches are coupled towards one scaling center in this example. This points out, that the SB-IGA formulation is not as computationally efficient as IGA for standard single-patch structures. However, the example shows a similar convergence rate for the IGA and the SB-IGA approach which validates the proposed method.
\begin{figure}[H]
\begin{minipage}{0.49\textwidth}
\begin{tikzpicture}
	\small \begin{semilogxaxis} [
		width = \textwidth,
		xlabel={\footnotesize Degrees of freedom},
		ylabel={\footnotesize $u_z/u_{z,ref}$},
		xmin=200, xmax=20000,
		ymin=0, ymax=1.2,
		xtick={10^3,10^4,10^5},
		ytick={0,0.2, 0.4, 0.6, 0.8, 1.0, 1.2},
		legend pos=south east,
		legend columns = 3,
		xmajorgrids=true,
		ymajorgrids=true,
		grid style=dashed,
		]
		\addplot
		coordinates {
		(540,-0.199134265765994/-0.3006) 
		(960,-0.287382260784672/-0.3006) 
		(1500,-0.297919759236538/-0.3006) 
		(2160,-0.299870894030744/-0.3006) 
		(2940,-0.300356028965457/-0.3006) 
		(3840,-0.300487722657382/-0.3006) 
		(4860,-0.300547197770159/-0.3006) 
		(6000,-0.300566575375962/-0.3006) 
		(7260,-0.300579804472187/-0.3006) 
		};
		\addlegendentry{$p=3$},
		\addplot
		coordinates {
		(960,-0.299712923427739/-0.3006) 
		(1815,-0.300544603380642/-0.3006) 
		(2940,-0.300580393454797/-0.3006) 
		(4335,-0.300590686691717/-0.3006) 
		(6000,-0.300591935740892/-0.3006) 
		(7935,-0.300592325630061/-0.3006) 
		(10140,-0.300592406875425/-0.3006) 
		(12615,-0.300592440027800/-0.3006) 
		(15360,-0.300592446353812/-0.3006) 
        };
		\addlegendentry{$p=4$},
		\addplot[color=darkgray,mark=triangle*]
		coordinates {
		(1500,-0.300620472818539/-0.3006) 
		(2940,-0.300592514512290/-0.3006) 
		(4860,-0.300592270225559/-0.3006) 
		(7260,-0.300592479622134/-0.3006) 
		(10140,-0.300592462938445/-0.3006) 
		(13500,-0.300592462971043/-0.3006) 
		(17340,-0.300592458789788/-0.3006) 
		(21660,-0.300592455064871/-0.3006) 
		(26460,-0.300592467915480/-0.3006) 
        };
		\addlegendentry{$p=5$},
		\addplot[color=black,mark=none]
		coordinates {
        (1,  1)
        (100000, 1)
        };
	\end{semilogxaxis}
\end{tikzpicture}
\caption*{\footnotesize{(a) Normalized displacement $u_{z}/u_{ref}$}}
\end{minipage}
\begin{minipage}{0.49\textwidth}
\begin{tikzpicture}
	\small \begin{loglogaxis} [
		width = \textwidth,
		xlabel={\footnotesize Degrees of freedom},
		ylabel={\footnotesize $1-|u_z/u_{z,ref}|$},
		xmin=60, xmax=20000,
		ymin=10^-11, ymax=1,
		xtick={10^3,10^4,10^5},
		ytick={10^-11,10^-10,10^-9,10^-8,10^-7,10^-6,10^-5,10^-4,10^-3,10^-2,10^-1,10^0},
		legend pos=south east,
		legend columns = 2,
		xmajorgrids=true,
		ymajorgrids=true,
		grid style=dashed,
		]
		\addplot
		coordinates {
		(540,1-0.199134265765994/0.300592456637501) 
		(960,1-0.287382260784672/0.300592456637501) 
		(1500,1-0.297919759236538/0.300592456637501) 
		(2160,1-0.299870894030744/0.300592456637501) 
		(2940,1-0.300356028965457/0.300592456637501) 
		(3840,1-0.300487722657382/0.300592456637501) 
		(4860,1-0.300547197770159/0.300592456637501) 
		(6000,1-0.300566575375962/0.300592456637501) 
		(7260,1-0.300579804472187/0.300592456637501) 
		};
		\addlegendentry{$p=3$},
		\addplot
		coordinates {
		(960,1-0.299712923427739/0.300592456637501) 
		(1815,1-0.300544603380642/0.300592456637501) 
		(2940,1-0.300580393454797/0.300592456637501) 
		(4335,1-0.300590686691717/0.300592456637501) 
		(6000,1.7329e-06) 
		(7935,4.35831e-07) 
		(10140,1.65547e-07) 
		(12615,5.52565e-08) 
		(15360,3.42114e-08) 
        };
		\addlegendentry{$p=4$},
		\addplot[color=blue,mark=o]
		coordinates {
		(75,1-0.230808441454189/0.300592456637501) 
		(147,1-0.279944686988844/0.300592456637501) 
		(363,1-0.300065498001738/0.300592456637501) 
		(1083,1-0.300584157111949/0.300592456637501) 
        (3675,4.17223e-07) 
		};
		\addlegendentry{$p=3$, IGA},
		\addplot[color=red,mark=square]
		coordinates {
		(108,1-0.272771400767432/0.300592456637501) 
		(192,1-0.300000190520429/0.300592456637501) 
		(432,1-0.300589846259662/0.300592456637501) 
        (1200,8.15676e-08) 
        (3888,3.07955e-10) 
        };
		\addlegendentry{$p=4$, IGA},
	\end{loglogaxis}
\end{tikzpicture}
\caption*{\footnotesize{(b) Comparison of $1-| u_{z}/u_{ref} | $ with \cite{Kiendl2009}}}
\end{minipage}
    \caption{\small On the left side, the convergence studies of the displacement $u_z$ at the middle of the straight edge with orders of $p = 3$, $p = 4$, and $p=5$ are shown. On the right side, the convergence rate is compared to the standard IGA discretization.}
   \label{fig:ScoLo_Results}
\end{figure}
\subsubsection*{Scordelis-Lo roof with one hole under gravity load}
For the incorporation of trimmed geometries, the previously defined Scordelis-Lo roof is subjected to trimming in the parametric domain of one hole of radius $r = 1$. Same as in the untrimmed roof, the structure is subjected to a gravity load of $g_z = 90$ per unit area. The reference solution of the vertical displacement at the midpoint of the straight, free edge as $u_{ref} = -0.361078869965661$ is obtained by overkill solution of the method for a Kirchhoff-Love shell proposed in \cite{Coradello2020} with order $p=5$ and $62.456$ active elements. The parametric mesh of domain $\tilde{\Omega} = [-4,4]^2$ and the initial shell configuration is shown in Fig. \ref{fig:ScoLo_1hole_Mesh}. The point of interest is marked in red in the figure.
\begin{figure}[H]
\centering
\scalebox{.9}{
\hspace*{-0.95cm}
\begin{minipage}{0.49\textwidth}
	\begin{figure}[H]
		\centering
			\begin{tikzpicture}
			\node (myfirstpic) at (0,0) {\includegraphics[width=\textwidth]{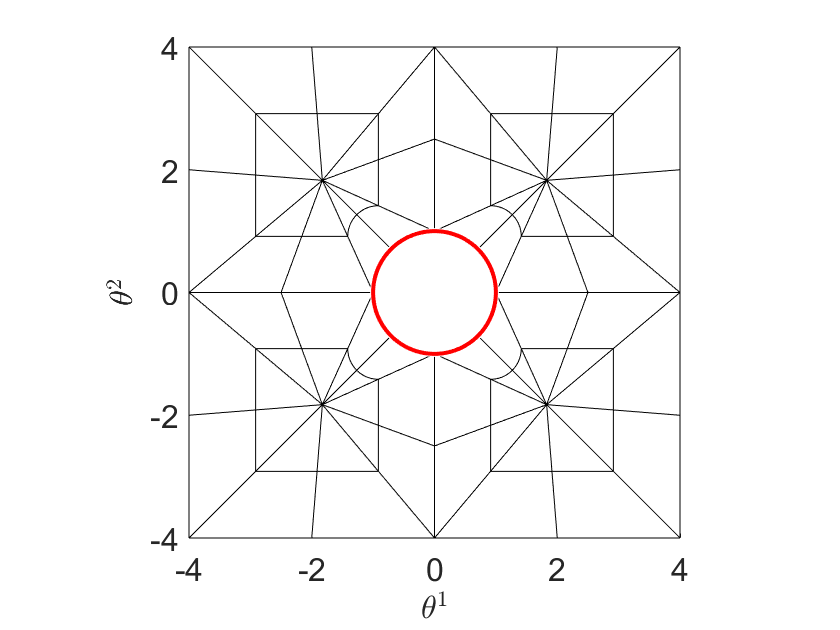}};
			\node at (0.15,2.7) {$\tilde{\o}$};
    \filldraw[red] (2.28,0.20) circle (2pt);
		\end{tikzpicture}
		\caption*{\footnotesize{(a) Parametric mesh}}
	\end{figure}
\end{minipage}\hspace*{-1cm}
\begin{minipage}{1cm}
\vspace{-3cm}
  \begin{tikzpicture}
      \draw[->,very thick, out=30,in=150] (-1,1) to (0,1);
      \node at (-0.5,1.4) {$\mathbf{R}$};
  \end{tikzpicture}
\end{minipage}
\begin{minipage}{0.49\textwidth}
	\begin{figure}[H]
		\centering
    \begin{tikzpicture}
	\node (myfirstpic) at (0,0) {\includegraphics[width=\textwidth]{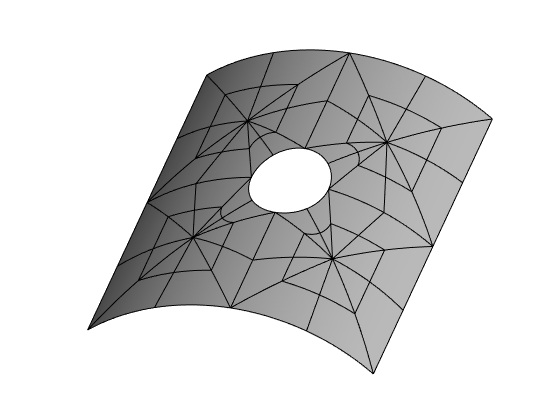}};
	\draw[|-|,thick] (1+0.2+0.2,-1.88-0.2-0.15) to (2.7+0.2+0.12,1.14-0.15+0.12);
	\node[rotate=65] at (2.4,-0.76) {$L=50$};
	\draw[thick,dashed] (-0.5,-3.1) -- (1.2,-1.88-0.25);
	 \draw [<->,thick] (-85:2.61)  arc (50:130:1) node [above,pos=0.45] {$\theta$};
	 \draw[thick,dashed] (-0.53,-3.1) -- (-2.55,-1.58);
  \filldraw[red] (2.04,0-0.47) circle (2pt);
	\end{tikzpicture}
		\caption*{\footnotesize{(b) Initial shell configuration}}
	\end{figure}
\end{minipage}}
    \centering
    \caption{\small Parametric mesh of the Scordelis-Lo roof with hole and the initial shell configuration of $h=1/2$.}
    \label{fig:ScoLo_1hole_Mesh}
\end{figure}
The results of the vertical deformation at the investigated point of the model with respect to the reference solution in Fig. \ref{fig:ScoLo1H_res} show good agreement. Accurate solutions are obtained even from the initial mesh. A slight underestimation is visible for the converged results, however, the difference is less than $0.5\%$ for $n=6$ and $p=5$. On the right side, the deformation plot for $p=5$ and $h=1/6$ is figured. 
\begin{figure}[H]
\begin{minipage}{0.49\textwidth}
\begin{tikzpicture}
	\small\begin{semilogxaxis} [
		width = \textwidth,
		xlabel={\footnotesize Degrees of freedom},
		ylabel={\footnotesize $u_z/u_{z,ref}$},
		xmin=500, xmax=50000,
		ymin=0.8, ymax=1.05,
		xtick={10^3,10^4,10^5},
		ytick={0.8,0.85,0.9,0.95, 1.0,1.05},
		legend pos=south east,
		legend columns = 3,
		xmajorgrids=true,
		ymajorgrids=true,
		grid style=dashed,
		]
		\addplot
		coordinates {
		(2268,-0.356057492203437/-0.361078869965661) 
		(4032,-0.358819580524420/-0.361078869965661) 
		(6300,-0.359297399764987/-0.361078869965661) 
		(9072,-0.359388295403222/-0.361078869965661) 
		(12348,-0.359415504898379/-0.361078869965661) 
		(16128,-0.359426540574916/-0.361078869965661) 
		(20412,-0.359431729751428/-0.361078869965661) 
		(25200,-0.359434478677752/-0.361078869965661) 
		(30492,-0.359436048979346/-0.361078869965661) 
		};
		\addlegendentry{$p=3$},
		\addplot
		coordinates {
		(4032,-0.359327511547032/-0.361078869965661) 
		(7623,-0.359427491995432/-0.361078869965661) 
		(12348,-0.359436493164539/-0.361078869965661) 
		(18207,-0.359438135353287/-0.361078869965661) 
		(25200,-0.359438665622433/-0.361078869965661) 
		(33327,-0.359438866142053/-0.361078869965661) 
		(42588,-0.359438971637159/-0.361078869965661) 
		(52983,-0.359438800473469/-0.361078869965661) 
		(64512,-0.359439145644459/-0.361078869965661) 
        };
		\addlegendentry{$p=4$},
		\addplot[color=darkgray,mark=triangle*]
		coordinates {
		(6300,-0.359435107994182/-0.361078869965661) 
		(12348,-0.359438744090236/-0.361078869965661) 
		(20412,-0.359438959704159/-0.361078869965661) 
		(30492,-0.359439020034924/-0.361078869965661) 
		(42588,-0.359439043169771/-0.361078869965661) 
        };
		\addlegendentry{$p=5$},
		\addplot[color=black,mark=none]
		coordinates {
        (1,  1)
        (100000, 1)
        };
	\end{semilogxaxis}
\end{tikzpicture}
\caption*{\footnotesize{(a) Normalized displacement $u_z / u_{ref}$}}
\end{minipage}
\begin{minipage}{0.5\textwidth}
\vspace{-1cm}
	\begin{figure}[H]
		\centering
  \begin{tikzpicture}
      \node (myfirstpic) at (0,0) {\includegraphics[width=\textwidth]{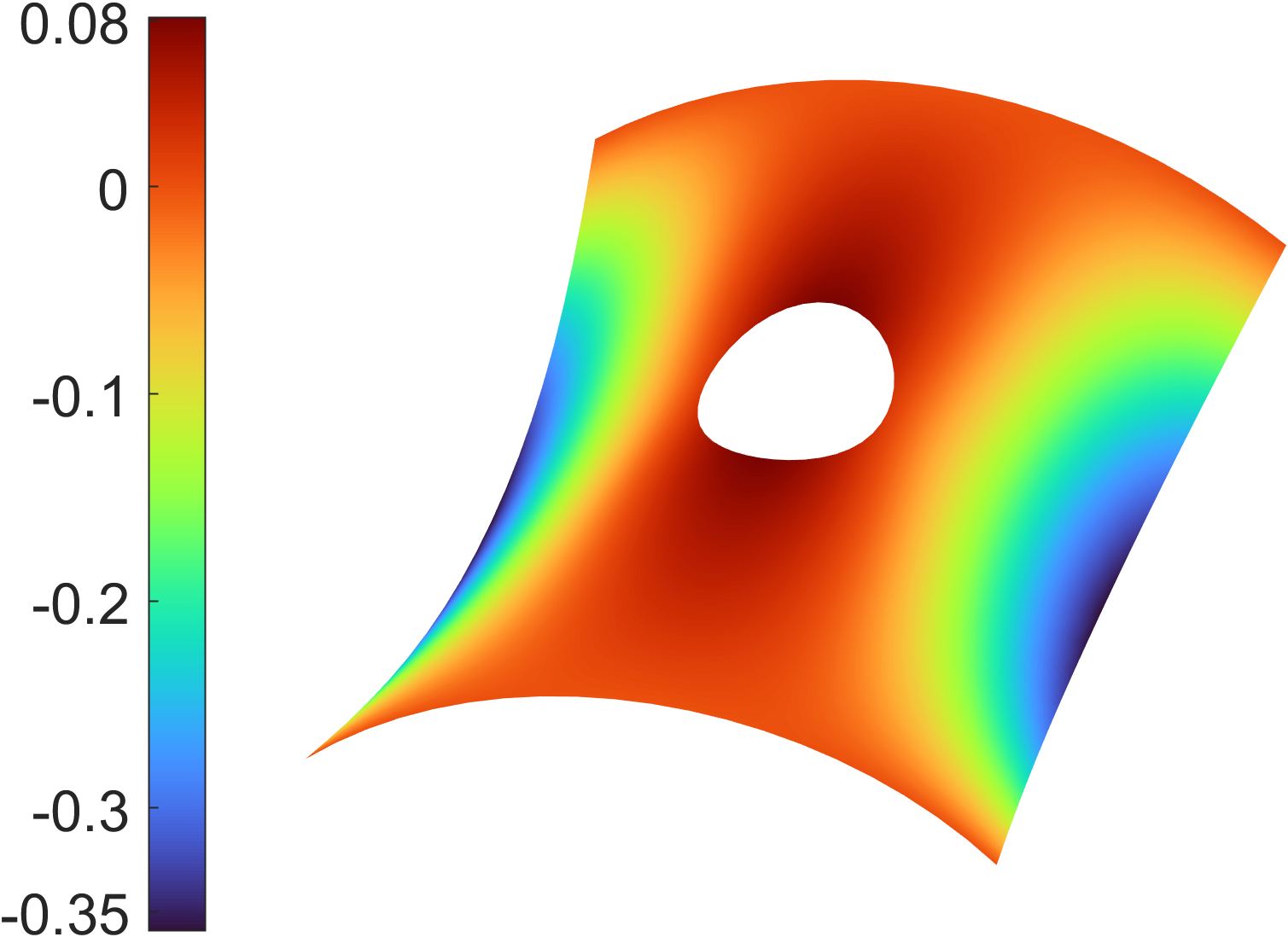}};
      \node at (-2.2,0) {$u_z$};
  \end{tikzpicture}
	\end{figure}
    \caption*{\footnotesize{(b) Deformation plot of component $u_z$ }}
\end{minipage}
    \caption{ \small On the left side, the convergence study on of the Scordelis-Lo roof with one hole under gravity load is figured. On the right side a plot of the deformation component $u_z$ is shown for a mesh of $h=1/6$ and $p=5$, $r=1$.}
   \label{fig:ScoLo1H_res}
\end{figure}
\subsubsection*{Scordelis-Lo roof with four holes under single load}
Further, the Scordelis-Lo roof is subjected to trimming in the parametric domain of four equally spaced holes of radius $r = 0.5$. The loading is changed to a point load in the vertical direction in the center of the structure of $F_z= 10^{-5}$. The example is shown in Fig. \ref{fig:Ex_5_Mesh} with the trimming curves shown in red and the point load applied in the center in green. 
\begin{figure}[H]
\centering
\scalebox{.9}{
\hspace*{-0.95cm}
\begin{minipage}{0.49\textwidth}
	\begin{figure}[H]
		\centering
			\begin{tikzpicture}
			\node (myfirstpic) at (0,0) {\includegraphics[width=\textwidth]{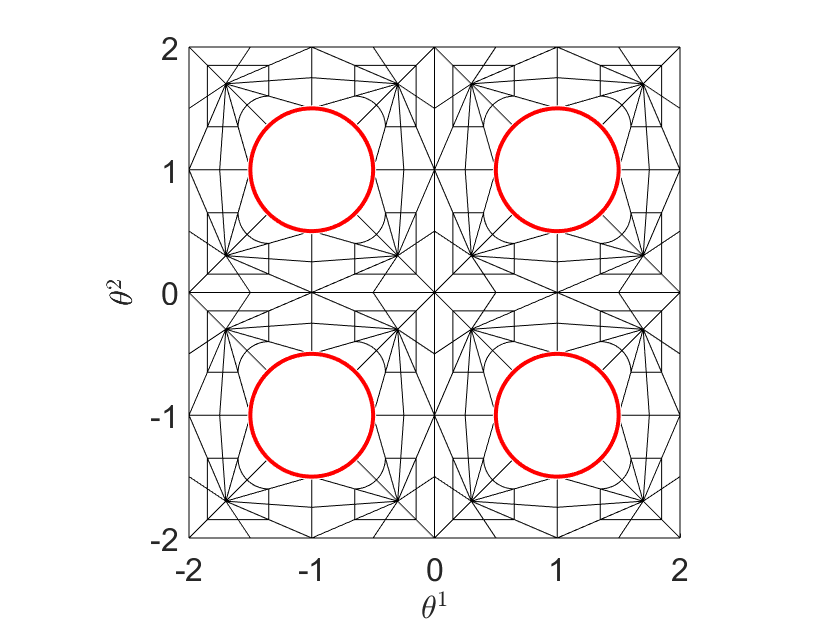}};
			\node at (0.15,2.7) {$\tilde{\o}$};
        \filldraw[green] (0.15,0.23) circle (2.2pt);
		\end{tikzpicture}
		\caption*{\footnotesize{(a) Parametric mesh}}
	\end{figure}
\end{minipage}\hspace*{-1cm}
\begin{minipage}{1cm}
\vspace{-3cm}
  \begin{tikzpicture}
      \draw[->,very thick, out=30,in=150] (-1,1) to (0,1);
      \node at (-0.5,1.4) {$\mathbf{R}$};
  \end{tikzpicture}
\end{minipage}
\begin{minipage}{0.49\textwidth}
	\begin{figure}[H]
		\centering
    \begin{tikzpicture}
	\node (myfirstpic) at (0,0) {\includegraphics[width=\textwidth]{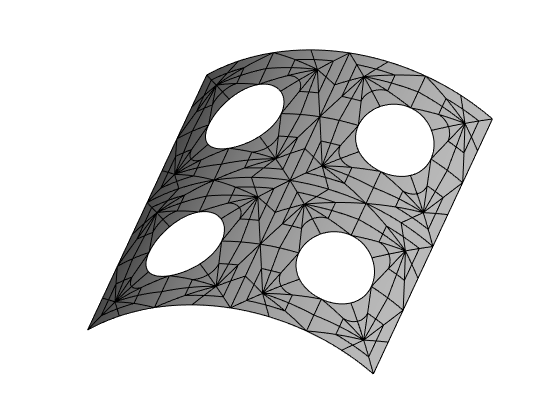}};
	\draw[|-|,thick] (1+0.2+0.2,-1.88-0.2-0.15) to (2.7+0.2+0.12,1.14-0.15+0.12);
	\node[rotate=65] at (2.4,-0.76) {$L=50$};
	\draw[thick,dashed] (-0.5,-3.1) -- (1.2,-1.88-0.25);
	 \draw [<->,thick] (-85:2.61)  arc (50:130:1) node [above,pos=0.45] {$\theta$};
	 \draw[thick,dashed] (-0.53,-3.1) -- (-2.55,-1.58);
    \draw[<-, line width=2pt, green] (0.13,0.4) -- (0.13, 1.4);
    \node[green, very thick] at (0.5,1.2) {$F_z$};
	\end{tikzpicture}
		\caption*{\footnotesize{(b) Initial shell configuration}}
	\end{figure}
\end{minipage}}
    \centering
    \caption{\small Mesh and parametric representation of the Scordelis-Lo roof with four holes ($h=1/2$).}
    \label{fig:Ex_5_Mesh}
\end{figure}
The results of the model can be compared to deformation plots presented in \cite[Figure 11 (a)]{Coradello2020}, which proposes an IGA shell with adaptive refinement. Fig. \ref{fig:Ex_5_Comp} shows the deformed structure with the displacement in $z$-direction for the SB-IGA approach of $p=3$, $r = 1$ and $h = 0.125$.
\vspace{-0.5cm}
\begin{figure}[H]
\begin{minipage}{0.49\textwidth}
	\begin{figure}[H]
		\centering
  \begin{tikzpicture}
      \node (myfirstpic) at (0,0) {\includegraphics[width=\textwidth]{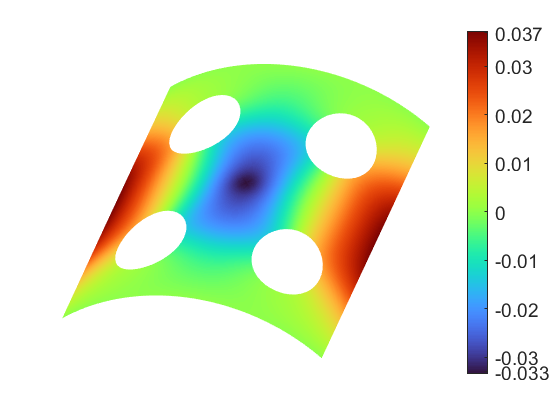}};
      \node at (4,0) {$u_z$};
  \end{tikzpicture}
	\end{figure}
\end{minipage}
    \centering 
    \caption{\small Deformed Scordelis-Lo roof with four holes with coloring of the $z$-displacement obtained for $p=3$, $h=1/6$, and $r=1$.}
    \label{fig:Ex_5_Comp}
\end{figure}
The results point out, that the approach derived herein is capable of computing the deformation for trimmed domains of shells. The deformation plot has its maximum and minimum at $u_{z,min} \approx -3.3 \cdot 10^{-1}$ and $u_{z,max} \approx 3.7 \cdot 10^{-1}$, respectively. These results are also obtained in \cite{Coradello2020}. Good agreement is visible on the whole domain including the area of the load application and the free edges.

\subsection{Hyperbolic paraboloid}
The model of the hyperbolic paraboloid is a saddle structure described in \cite{Bathe2000}. Besides its complexity due to the double-curved structure and the negative curvature, the problem is compared to the ASG$^1$ approach for standard IGA multi-patches. The structure is clamped one-sided and loaded by a uniformly distributed load of $g_z = -8000 \cdot t$. The properties are length $L = 1$, Young's modulus $E=2\cdot 10^{11}$, Poisson's ratio $\nu = 0.3$ and thickness of $t=1/1000$ and $t = 1/100$. The exact parametrization function of both hyperbolic paraboloid shells is chosen as
\begin{equation}
     \mathbf{R}(\theta^1,\theta^2) = \left[\theta^1 \qquad \theta^2 \qquad \left(\theta^1\right)^2-\left(\theta^2\right)^2 \right]^T.
\end{equation}
The reference solution is obtained from \cite{Bathe2000} as the vertical deflection at the middle of the free, curved edge with $u_{z,t=1/1000} = -0.0063941$ and $u_{z,t=1/100} = -0.93137\cdot 10^{-4}$. The partition of the parametric mesh of domain $\tilde{\Omega} = [-0.5,0.5]^2$ and the initial shell configuration is shown in Fig. \ref{fig:HypPara_Mesh} with the clamped boundary highlighted in blue and the point of interest highlighted in red. The scaling center is placed with an offset in $\theta^2$-direction as $\theta^2_{off} = -0.1$. This yields the possibility to have a higher mesh density towards the Dirichlet boundary conditions naturally enforced by moving the scaling center. 
\begin{figure}[H]
\centering
\scalebox{.9}{
\hspace*{-0.95cm}
\begin{minipage}{0.49\textwidth}
	\begin{figure}[H]
		\centering
			\begin{tikzpicture}
			\node (myfirstpic) at (0,0) {\includegraphics[width=\textwidth]{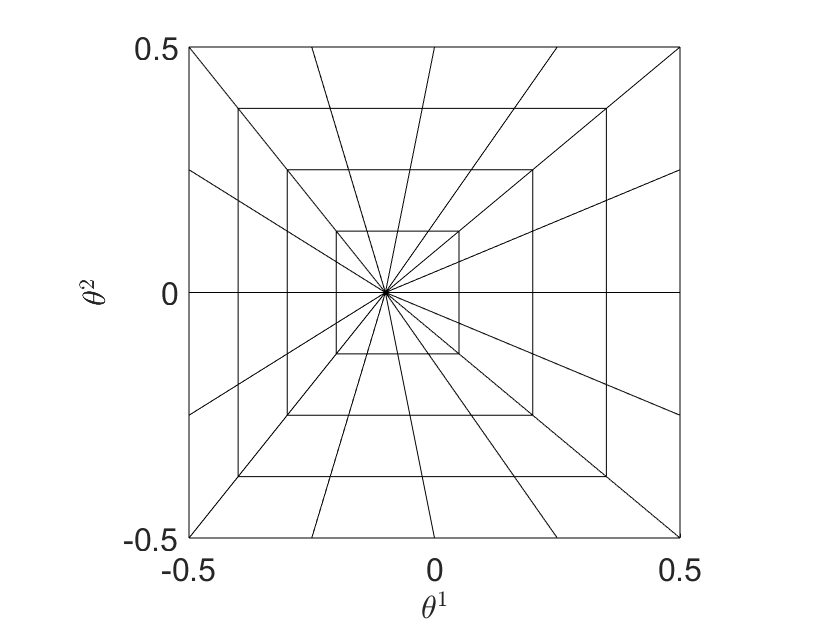}};
			\node at (0.15,2.7) {$\tilde{\o}$};
            \draw[blue,very thick] (-2.05,-1.945) to (-2.05,2.35);
            \filldraw[red] (2.29,0.21) circle (2pt);
		\end{tikzpicture}
		\caption*{\footnotesize{(a) Parametric mesh}}
	\end{figure}
\end{minipage}\hspace*{-1cm}
\begin{minipage}{1cm}
\vspace{-3cm}
  \begin{tikzpicture}
      \draw[->,very thick, out=30,in=150] (-1,1) to (0,1);
      \node at (-0.5,1.4) {$\mathbf{R}$};
  \end{tikzpicture}
\end{minipage}
\begin{minipage}{0.49\textwidth}
	\begin{figure}[H]
		\centering
    \begin{tikzpicture}
	\node (myfirstpic) at (0,0) {\includegraphics[width=\textwidth]{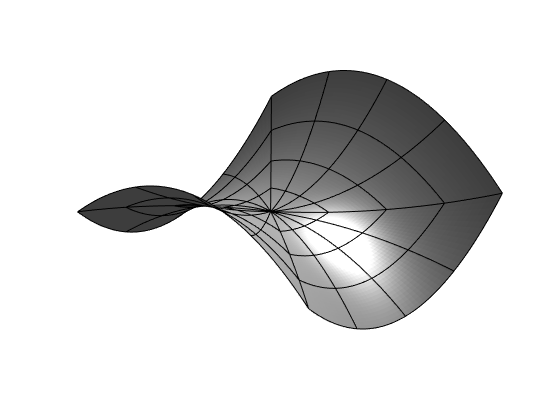}};
	\draw[|-|,thick] (0.3,-1.20-1) to (2.7+0.25,0.2-1);
    \draw[|-|,thick] (0.3,-1.20-1) to (-2.64,0.0-1);
	\node[rotate=32] at (1.7,-0.76-1) {$L=1$};
    \node[rotate=-23] at (-1.5,-0.76-1) {$L=1$};
    \filldraw[red] (-1.05,0.15) circle (2pt);
    \draw [blue,very thick] plot [smooth] coordinates {(-0.18,1.5) (0.5,1.8)  (1.25,1.8) (1.98,1.4) (2.97,0.2)}; 
	\end{tikzpicture}
		\caption*{\footnotesize{(b) Initial shell configuration}}
	\end{figure}
\end{minipage}}
    \centering
    \caption{\small Mesh and parametric representation of the hyperbolic paraboloid. On the left, the parametric mesh for $h=1/4$ is shown. On the right, the corresponding initial shell configuration is plotted. The clamped boundary is marked with a blue line and the point of interest for the evaluation is marked with a red dot in both figures.}
    \label{fig:HypPara_Mesh}
\end{figure}
The results of the convergence study of both thicknesses are shown in Fig. \ref{fig:ParaBolo_Res}. The computed results converge smoothly towards the reference solution for all polynomial orders of the basis functions. For the thin shell, the computation is converging less rapidly than for the thick shell. 
\begin{figure}[H]
\begin{minipage}{0.49\textwidth}
\begin{tikzpicture}
	\small \begin{semilogxaxis} [
		width = \textwidth,
		xlabel={\footnotesize Degrees of freedom},
		ylabel={\footnotesize $u_z/u_{z,ref}$},
		xmin=500, xmax=50000,
		ymin=0.2, ymax=1.1,
		xtick={10^3,10^4,10^5},
		ytick={0,0.2, 0.4, 0.6, 0.8, 1.0, 1.2},
		legend pos=south east,
		legend columns = 3,
		xmajorgrids=true,
		ymajorgrids=true,
		grid style=dashed,
		]
		\addplot
		coordinates {
		(540,-7.732771490412845e-04/-0.0063941) 
		(960,-0.001789482448689/-0.0063941) 
		(1500,-0.002970692239703/-0.0063941) 
		(2160,-0.004129769567185/-0.0063941) 
		(2940,-0.004780918925169/-0.0063941) 
		(3840,-0.005155343714999/-0.0063941) 
		(4860,-0.005392891940006/-0.0063941) 
		(6000,-0.005564929280986/-0.0063941) 
		(7260,-0.005697251488165/-0.0063941) 
		(8640,-0.005802438717773/-0.0063941) 
		(10140,-0.005886914226463/-0.0063941) 
		(11760,-0.005955703518588/-0.0063941) 
		(13500,-0.006012046032889/-0.0063941) 
		(15360,-0.006058660980425/-0.0063941) 
		(17340,-0.006097425205839/-0.0063941) 
        (19440,-0.006129929225940/-0.0063941) 
		(21660,-0.006157331512144/-0.0063941) 
		(24000,-0.006180605251743/-0.0063941) 
		(26460,-0.006200486599625/-0.0063941) 
		(29040,-0.006217585742579/-0.0063941) 
		(31740,-0.006232378738388/-0.0063941) 
		(34560,-0.006245255490631/-0.0063941) 
        (37500,-0.006256528811580/-0.0063941) 
		};
		\addlegendentry{$p=3$},
		\addplot
		coordinates {
		(960,-0.002950331475109/-0.0063941) 
		(1815,-0.004751152652467/-0.0063941) 
		(2940,-0.005375955670491/-0.0063941) 
		(4335,-0.005665648707705/-0.0063941) 
		(6000,-0.005846251757035/-0.0063941) 
		(7935,-0.005967908327856/-0.0063941) 
		(10140,-0.006055229992996/-0.0063941) 
		(12615,-0.006119559376012/-0.0063941) 
		(15360,-0.006168249020427/-0.0063941) 
		(18375,-0.006205553262372/-0.0063941) 
		(21660,-0.006234561319570/-0.0063941) 
		(25215,-0.006257313323045/-0.0063941) 
		(29040,-0.006275380827907/-0.0063941) 
		(33135,-0.006289874995586/-0.0063941) 
		(37500,-0.006301652844844/-0.0063941) 
        };
		\addlegendentry{$p=4$},
		\addplot[color=darkgray,mark=triangle*]
		coordinates {
		(1500,-0.004971217989138/-0.0063941) 
		(2940,-0.005653459741511/-0.0063941) 
		(4860,-0.005929151598975/-0.0063941) 
		(7260,-0.006077773402105/-0.0063941) 
		(10140,-0.006165365458615/-0.0063941) 
		(13500,-0.006220734925224/-0.0063941) 
		(17340,-0.006257873963323/-0.0063941) 
		(21660,-0.006283896836440/-0.0063941) 
		(26460,-0.006302745179638/-0.0063941) 
        };
		\addlegendentry{$p=5$},
		\addplot[color=black,mark=none]
		coordinates {
        (1,  1)
        (100000, 1)
        };
	\end{semilogxaxis}
\end{tikzpicture}
\caption*{\footnotesize{$t=1/1000$}}
\end{minipage}
\begin{minipage}{0.49\textwidth}
\begin{tikzpicture}
	\small \begin{semilogxaxis} [
		width = \textwidth,
		xlabel={\footnotesize Degrees of freedom},
		ylabel={\footnotesize $u_z/u_{z,ref}$},
		xmin=500, xmax=50000,
		ymin=0.2, ymax=1.1,
		xtick={10^3,10^4,10^5},
		ytick={0,0.2, 0.4, 0.6, 0.8, 1.0, 1.2},
		legend pos=south east,
		legend columns = 3,
		xmajorgrids=true,
		ymajorgrids=true,
		grid style=dashed,
		]
		\addplot
		coordinates {
		(540,-3.534389360239380e-05/-0.93137*10^-4) 
		(960,-6.164320570408263e-05/-0.93137*10^-4) 
		(1500,-7.371062220624791e-05/-0.93137*10^-4) 
		(2160,-8.024134045254622e-05/-0.93137*10^-4) 
		(2940,-8.399115985900662e-05/-0.93137*10^-4) 
		(3840,-8.628287291796832e-05/-0.93137*10^-4) 
		(4860,-8.776472649973871e-05/-0.93137*10^-4) 
		(6000,-8.877302713164594e-05/-0.93137*10^-4) 
		(7260,-8.949397123916558e-05/-0.93137*10^-4) 
		(8640,-9.002759457529679e-05/-0.93137*10^-4) 
		(10140,-9.043746816677340e-05/-0.93137*10^-4) 
		(11760,-9.075938287391117e-05/-0.93137*10^-4) 
		(13500,-9.101913002036965e-05/-0.93137*10^-4) 
		(15360,-9.123179523099181e-05/-0.93137*10^-4) 
		(17340,-9.140943730856496e-05/-0.93137*10^-4) 
        (19440,-9.155929893807210e-05/-0.93137*10^-4) 
		(21660,-9.168767679040022e-05/-0.93137*10^-4) 
		(24000,-9.179841868074240e-05/-0.93137*10^-4) 
		(26460,-9.189510618466120e-05/-0.93137*10^-4) 
		(29040,-9.197995166461474e-05/-0.93137*10^-4) 
		(31740,-9.205513366952135e-05/-0.93137*10^-4) 
		(34560,-9.212200622558474e-05/-0.93137*10^-4) 
        (37500,-9.218196699786979e-05/-0.93137*10^-4) 
		};
		\addlegendentry{$p=3$},
		\addplot
		coordinates {
		(960,-7.316968183582717e-05/-0.93137*10^-4) 
		(1815,-8.270919825524009e-05/-0.93137*10^-4) 
		(2940,-8.701286430353728e-05/-0.93137*10^-4) 
		(4335,-8.907554567301876e-05/-0.93137*10^-4) 
		(6000,-9.018887616638319e-05/-0.93137*10^-4) 
		(7935,-9.085080185827809e-05/-0.93137*10^-4) 
		(10140,-9.128286540173374e-05/-0.93137*10^-4) 
		(12615,-9.158486634357766e-05/-0.93137*10^-4) 
		(15360,-9.180735366188653e-05/-0.93137*10^-4) 
		(18375,-9.197775536918955e-05/-0.93137*10^-4) 
		(21660,-9.211228139470124e-05/-0.93137*10^-4) 
		(25215,-9.222103402370418e-05/-0.93137*10^-4) 
		(29040,-9.231066337995258e-05/-0.93137*10^-4) 
		(33135,-9.238571327742426e-05/-0.93137*10^-4) 
		(37500,-9.244940028145973e-05/-0.93137*10^-4) 
        };
		\addlegendentry{$p=4$},
		\addplot[color=darkgray,mark=triangle*]
		coordinates {
		(1500,-8.444098913758520e-05/-0.93137*10^-4) 
		(2940,-8.883443865593166e-05/-0.93137*10^-4) 
		(4860,-9.050025741601437e-05/-0.93137*10^-4) 
		(7260,-9.128280020694893e-05/-0.93137*10^-4) 
		(10140,-9.172293122120323e-05/-0.93137*10^-4) 
		(13500,-9.200281452251640e-05/-0.93137*10^-4) 
		(17340,-9.219584760730372e-05/-0.93137*10^-4) 
		(21660,-9.233663610152000e-05/-0.93137*10^-4) 
		(26460,-9.244357805155399e-05/-0.93137*10^-4) 
        };
		\addlegendentry{$p=5$},
		\addplot[color=black,mark=none]
		coordinates {
        (1,  1)
        (100000, 1)
        };
	\end{semilogxaxis}
\end{tikzpicture}
\caption*{\footnotesize{$t=1/100$}}
\end{minipage}
    \caption{\small Example of the hyperbolic paraboloid. On both sides the convergence studies for orders of $p=3$, $p=4$, and $p=5$ with $r=1$ are shown for the thin shell (left) and the thick shell (right).}
   \label{fig:ParaBolo_Res}
\end{figure}

\begin{figure}[H]
    \centering
\begin{minipage}{0.49\textwidth}
\begin{tikzpicture}
	\small \begin{semilogxaxis} [
		width = \textwidth,
		xlabel={\footnotesize Degrees of freedom},
		ylabel={\footnotesize $u_z/u_{z,ref}$},
		xmin=500, xmax=50000,
		ymin=0.2, ymax=1.1,
		xtick={10^3,10^4,10^5},
		ytick={0,0.2, 0.4, 0.6, 0.8, 1.0, 1.2},
		legend pos=south east,
		legend columns = 2,
		xmajorgrids=true,
		ymajorgrids=true,
		grid style=dashed,
		]
		\addplot
		coordinates {
		(540,-3.534389360239380e-05/-0.93137*10^-4) 
		(1500,-7.371062220624791e-05/-0.93137*10^-4) 
		(4860,-8.776472649973871e-05/-0.93137*10^-4) 
		(17340,-9.140943730856496e-05/-0.93137*10^-4) 
        (37500,-9.218196699786979e-05/-0.93137*10^-4) 
		};
		\addlegendentry{$p=3$},
		\addplot
		coordinates {
		(960,-7.316968183582717e-05/-0.93137*10^-4) 
		(2940,-8.701286430353728e-05/-0.93137*10^-4) 
		(10140,-9.128286540173374e-05/-0.93137*10^-4) 
		(37500,-9.244940028145973e-05/-0.93137*10^-4) 
        };
		\addlegendentry{$p=4$},
        \addplot[color=blue,mark=o]
        coordinates {
        (505, 0.476732585)
        (2211, 0.744890925)
        (9076, 0.887860836)
        (36691, 0.954095118)
        };
        \addlegendentry{$p=3$, SAS},
        \addplot[color=red,mark=square]
        coordinates {
        (660, 0.656387403)
        (2443, 0.847854991)
        (9540, 0.936313635)
        (37304, 0.972265541)
        };
        \addlegendentry{$p=4$, SAS},
		\addplot[color=black,mark=none]
		coordinates {
        (1,  1)
        (100000, 1)
        };
	\end{semilogxaxis}
\end{tikzpicture}
\caption*{\footnotesize{(a) $t=1/100$}}
\end{minipage}
\begin{minipage}{0.45\textwidth}
\vspace{-1cm}
		\centering
        \begin{tikzpicture}
      \node (myfirstpic) at (0,0) {\includegraphics[width=\textwidth]{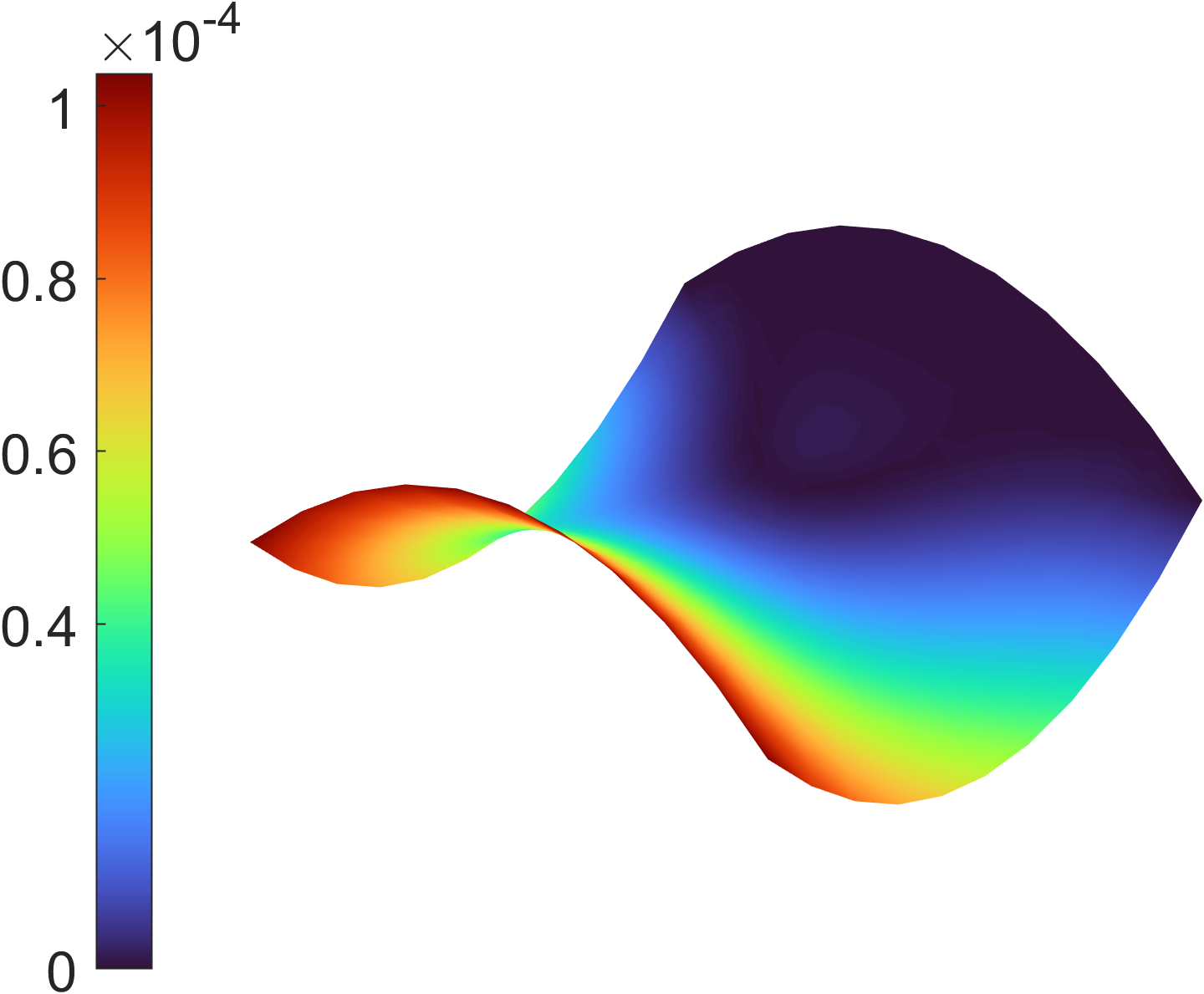}};
      \node at (-2.2,0) {$|\mathbf{u}|$};
    \end{tikzpicture}
		\caption*{\footnotesize{(b) Deformation magnitude $t=1/100$}}
\end{minipage}
\caption{\small Comparison of the results from the proposed approach with the standard analysis-suitable $G^1$ for Kirchhoff-Love shells proposed in \cite{Farahat2022} for a thickness of $t= 1/100$ (left) and an exemplary deformation plot of the overall deformation $u$ for $p=3$ and $h=1/24$ (right).}
    \label{fig:CompFH}
\end{figure}
Further, the proposed approach is compared to the analysis-suitable $G^1$ approach for standard IGA patches. Fig. \ref{fig:CompFH} shows a comparison of the results for the order of $p=3$ and $p=4$. It is important to note, that the approach herein has a regularity of $r=1$ for various orders, while the compared approach has a regularity of $r=p-2$. The comparison shows, that the Kirchhoff-Love shell in boundary representation seems to converge faster towards the reference solution than the approximate solutions from \cite{Farahat2022} utilizing non-degenerate IGA multi-patches.

\subsection{Flat shell with multiple holes}
To take the trimming more into account, a flat shell structure with an asymmetrical arrangement is evaluated following \cite{man2013}. The model is presented in Fig. \ref{fig:MultHole_Mesh} with the parametric mesh and the corresponding deformation. The quadratic shell of side length $L = 5$ and thickness $t = 0.005$ is trimmed by 16 circular holes of radius $r=0.025$ and clamped at all four sides. The material parameters are $E = 8.736\cdot 10^{7}$ and $\nu = 0.3$. The structure is loaded by a vertical load of $g_z = \sin\left(\pi x \right) \cdot \sin\left(\pi y \right)$.
\begin{figure}
\centering
\scalebox{.9}{
\hspace*{-0.95cm}
\begin{minipage}{0.49\textwidth}
	\begin{figure}[H]
		\centering
			\begin{tikzpicture}
			\node (myfirstpic) at (0,0) {\includegraphics[width=\textwidth]{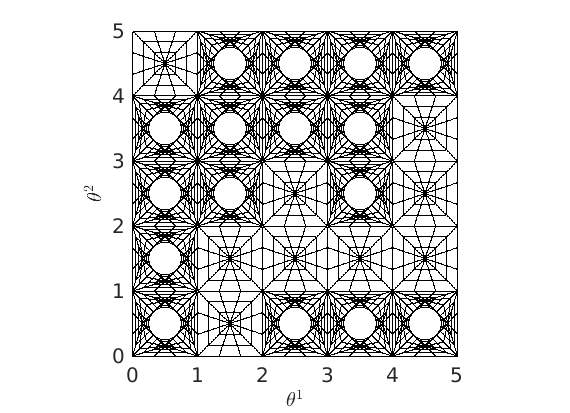}};
			\node at (0.15,2.7) {${\o}$};
		\end{tikzpicture}
		\caption*{\footnotesize{(a) Parametric mesh}}
	\end{figure}
\end{minipage}\hspace*{-1cm}
\begin{minipage}{1cm}
\vspace{-3cm}
  \begin{tikzpicture}
  \end{tikzpicture}
\end{minipage}
\begin{minipage}{0.49\textwidth}
	\begin{figure}[H]
		\centering
    \begin{tikzpicture}
	\node (myfirstpic) at (0,0) {\includegraphics[width=\textwidth]{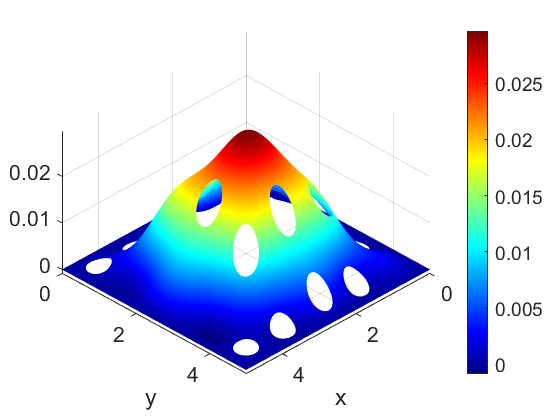}};
    \node at (2.2,0.0) {$u_z$};
	\end{tikzpicture}
		\caption*{\footnotesize{(b) Deformation plot}}
	\end{figure}
\end{minipage}}
    \centering
    \caption{\small Parametric representation of the flat shell with multiple holes and an exemplary deformation plot with $p=3$ and $h=1/3$.}
    \label{fig:MultHole_Mesh}
\end{figure}
To compare the results to reference solutions from the literature  \cite{man2013}, a normalization of the vertical deformation is done as
\begin{equation}
    u_z^* = \frac{u_z D}{L^4},
\end{equation}
with $u_z$ as the computed deflection in $z$-direction, $D$ the flexural stiffness and $L$ the side length of the shell. Fig. \ref{fig:MultComp} shows the results of the proposed approach for the deformation $u_z$ and to compare the results to the reference solution from the literature, the results are evaluated in terms of a normalization. We refer to the reference solution of an SB-FEM plate formulation with scaling in the thickness direction and semi-analytic solution procedure presented in \cite[Figure 29 (a)]{man2013}. Good agreement is visible for the proposed approach considering maximal normalized deflection and the contour lines.
\begin{figure}[H]
\hspace*{-0.95cm}
\begin{minipage}{0.44\textwidth}
	\begin{figure}[H]
		\begin{tikzpicture}
            \node (myfirstpic) at (0,0) {\includegraphics[width=\textwidth]{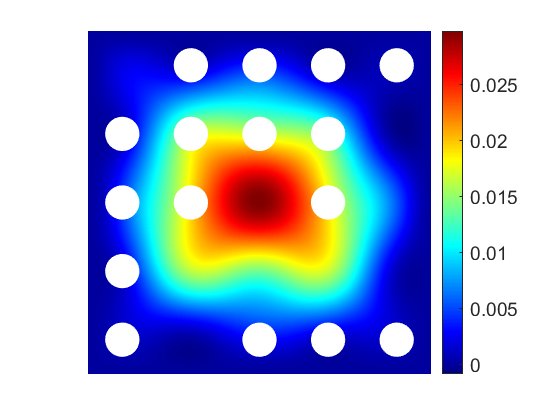}};
            \node at (3.2,0) {$u_z$};
    	\end{tikzpicture}
		\caption*{\footnotesize{(a) Proposed approach with $p=3$ and $h=1/3$}}
	\end{figure}
\end{minipage}\hspace*{-1cm}
\begin{minipage}{1cm}
\vspace{-3cm}
  \begin{tikzpicture}
  \end{tikzpicture}
\end{minipage}
\begin{minipage}{0.45\textwidth}
	\begin{figure}[H]
    \begin{tikzpicture}
\node (myfirstpic) at (0,0) {\includegraphics[width=\textwidth]{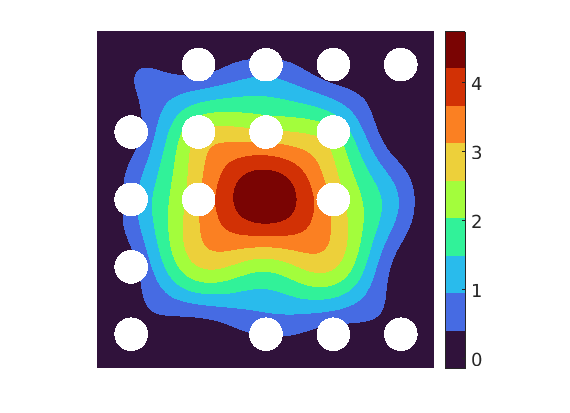}}; 
\node at (2.85,0) {$u_z^*$};
	\end{tikzpicture}
		\caption*{\footnotesize{(b) Normalized displacement $u_z^*$ of (a)}}
	\end{figure}
\end{minipage}
    \centering
    \caption{\small Left, the deformation $u_z$ is shown for the flat shall with multiple holes. Right, the normalized vertical displacements $u_z^*$ of the proposed method look similar to results by an SBFEM formulation using normal scaling presented in \cite{man2013}.}
    \label{fig:MultComp}
\end{figure}

\subsection{Violin}
In the last example, we utilize the boundary representation in the most sophisticated extent by evaluating a sketch of a violin following the idea of \cite{Coradello2020hierarchically} with modifications for the boundary conditions and the parametrization functions. The violin in this contribution is clamped at the whole outer boundary and trimmed by two so-called "F-holes". It is loaded by a constant vertical dead load of $g_z = 0.05$ per unit area. The parameters are $t=0.25$, $E = 10^5$ and $\nu = 0.1$. The underlying parametrization function of the violin is
\begin{equation}
    \mathbf{R}(\theta^1, \theta^2) = \left[ \theta^1 \qquad \theta^2 \qquad 2\exp\left(-0.0025\left(\theta^1\right)^2 \right)\cdot 2\exp\left(-0.01\left(\theta^2\right)^2 \right) \right]. 
\end{equation}
The parametric mesh and the corresponding initial shell configuration are shown in Fig. \ref{fig:Violin_Mesh}. The advantages of the scaled boundary approach are especially visible due to the various shapes of the SB-patches including triangles, quadrilaterals, and pentagons. Further, the outer boundary curve of the violin contains vertices that are easily included in the computational domain.
Even though there is no reference solution, the results seem reasonable as the deformation is symmetric in the $y$-axis, and the boundary conditions are fulfilled. 
\vspace{-0.8cm}
\begin{figure}[H]
\centering
\scalebox{1}{
\hspace*{-0.95cm}
\begin{minipage}{0.49\textwidth}
	\begin{figure}[H]
		\centering
			\begin{tikzpicture}
            \node (myfirstpic) at (0-1,0) {\includegraphics[width=\textwidth]{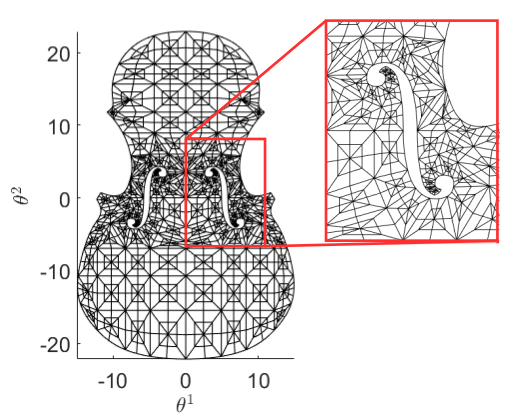}};
			\node at (0.15-1-1.1,2.7+0.1) {$\tilde{\o}$};
		\end{tikzpicture}
		\caption*{\footnotesize{(a) Parametric mesh}}
	\end{figure}
\end{minipage}\hspace*{-1cm}
\begin{minipage}{1cm}
\vspace{-3cm}
  \begin{tikzpicture}
  \end{tikzpicture}
\end{minipage}
\begin{minipage}{0.49\textwidth}
	\begin{figure}[H]
		\centering
    \begin{tikzpicture}
	\node (myfirstpic) at (0,0) {\includegraphics[width=\textwidth]{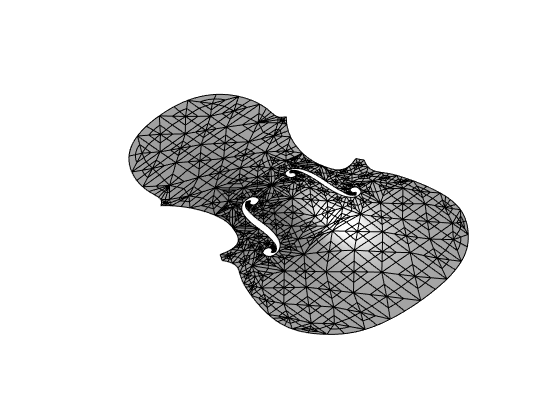}};
    \draw[->,very thick, out=30,in=150] (-1-2.5,1) to (0-2.5,1);
    \node at (-3,1.4) {$\mathbf{R}$};
	\end{tikzpicture}
		\caption*{\footnotesize{(b) Initial shell configuration}}
	\end{figure}
\end{minipage}}
    \centering
    \caption{\small Parametric representation and initial shell configuration of the violin. On the left, the parametric mesh for $h=1/2$ is shown. On the right, the corresponding initial shell configuration is shown. The violin is clamped at the outer boundary in all directions including rotations.}
    \label{fig:Violin_Mesh}
\end{figure}
\begin{figure}[H]
\vspace{-1.4cm}
\hspace*{-0.95cm}
\begin{minipage}{0.4\textwidth}
	\begin{figure}[H]
			\begin{tikzpicture}
            \node (myfirstpic) at (0,0) {\includegraphics[width=\textwidth]{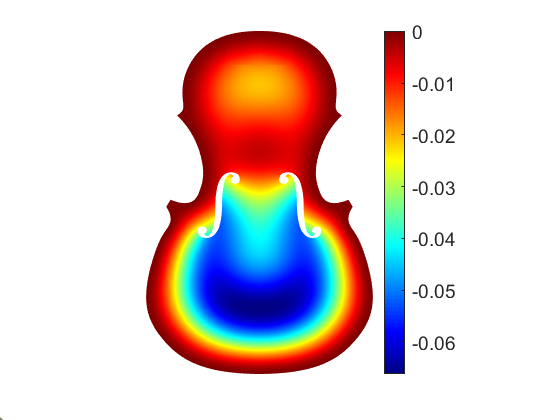}};
            \node at (2.3,0) {$u_z$};
    	\end{tikzpicture}
	\end{figure}
\end{minipage}\hspace*{-1cm}
\begin{minipage}{1cm}
  \begin{tikzpicture}
  \end{tikzpicture}
\end{minipage}
\begin{minipage}{0.55\textwidth}
	\begin{figure}[H]
    \begin{tikzpicture}
    \node (myfirstpic) at (0,0) {\includegraphics[width=\textwidth]{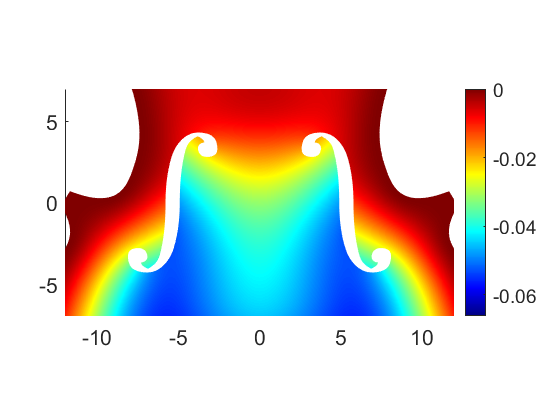}}; 
    \node at (2.3,1) {$u_z$};
	\end{tikzpicture}
	\end{figure}
\end{minipage}
\vspace{-1.0cm}
    \centering
    \caption{\small Deformation plot of the proposed approach for a violin with $p=3$ and $h=1/2$.}
    \label{fig:Ex_5_Mesh}
\end{figure}

\section{Conclusions}\label{Section:ConOut}
In this contribution, a Kirchhoff-Love shell was derived in the framework of scaled boundary isogeometric analysis. To ensure $C^1$-continuity across patches, a coupling of the scaled boundary geometries was applied that fits the concept of analysis-suitable $G^1$  parametrizations from \cite{Collin2016} and a special treatment of the basis functions in the vicinity of the scaling center was implemented. The benefits of the boundary representation technique are pointed out by the incorporation of trimming curves as part of the geometry. The shell formulation is at first tested in general against optimal convergence rates of the $H^2$ seminorm and the $L^2$ norm. Further, a comparison to the Kirchhoff-Love shell proposed by \cite{Kiendl2009} and in the terms of multi-patch structures \cite{Farahat2022} is done. The application of trimming is presented in several examples and good agreement with examples from the literature is obtained. The method is especially powerful when it comes to multi-patch geometries that cannot easily be described by a single IGA patch which is outlined by an example of a violin structure with trimmed F-holes. \\
In further work, an exact representation of the shells in the physical domain is pending. A procedure to adjust the planar domain and the corresponding parametrization according to the physical domain will improve the applicability of the approach. Besides, the formulation of the Kirchhoff-Love shell may be extended to geometric and material nonlinearity to show the formulations' power not only for the linear case but to make it applicable to more advanced shell problems including buckling or plastic deformations.


\paragraph{Declarations}
\vspace{-0.1cm}
\paragraph{Funding} This study was funded by the DFG (German Research Foundation) under Grant No. KL1345/10-2 (project number: 667493) and Grant No. SI756/5-2 (project number: 667494).
\vspace{-0.2cm}
\paragraph{Conflict of Interests} The authors declare that they have no conflict of interest.
\vspace{-0.4cm}
\bibliographystyle{IEEEtran}
\bibliography{Manuscript}

\end{document}